\documentclass[aap]{imsart}

\arxiv{arXiv:1609.02901}

\RequirePackage[OT1]{fontenc}
\RequirePackage{amsthm,amsmath}
\RequirePackage[numbers]{natbib}

\startlocaldefs
\numberwithin{equation}{section}
\theoremstyle{plain}
\newtheorem{thm}{Theorem}[section]
\endlocaldefs

\def\[#1\]{\begin{align}#1\end{align}}
\def\*[#1\]{\begin{align*}#1\end{align*}}

\newcommand{\Reals}{\mathbb{R}}
\newcommand{\Nats}{\mathbb{N}}

\newcommand{\PosReals}{\Reals_{> 0}}

\newcommand{\dee}{\mathrm{d}}

\DeclareMathOperator*{\newlim}{\mathrm{lim}\vphantom{\mathrm{infsup}}}
\DeclareMathOperator*{\newmin}{\mathrm{min}\vphantom{\mathrm{infsup}}}
\DeclareMathOperator*{\newmax}{\mathrm{max}\vphantom{\mathrm{infsup}}}
\DeclareMathOperator*{\newinf}{\mathrm{inf}\vphantom{\mathrm{infsup}}}
\DeclareMathOperator*{\newsup}{\mathrm{sup}\vphantom{\mathrm{infsup}}}
\renewcommand{\lim}{\newlim}
\renewcommand{\min}{\newmin}
\renewcommand{\max}{\newmax}
\renewcommand{\inf}{\newinf}
\renewcommand{\sup}{\newsup}

\newcommand{\cF}{\mathcal F}

\newcommand{\BorelSets}[1]{\mathcal{B}[#1]}
\newcommand{\NSE}[1]{{^{*}#1}}
\newcommand{\ST}{\mathsf{st}}

\newcommand{\HReals}{\NSE{\Reals}}

\newcommand{\NS}[1]{\mathrm{NS}(#1)}

\newcommand{\cD}{\mathcal{D}}
\newcommand{\cQ}{\mathcal{Q}}
\newcommand{\cA}{\mathcal{A}}

\newtheorem{open problem}{Open Problem}

\newcommand{\Loeb}[1]{\overline{#1}}

\newcommand{\interior}[1]{%
  {\kern0pt#1}^{\mathrm{o}}%
}

\newcommand{\refproof}[1]{See \cref{#1} for \IfSubStr{#1}{,}{proofs}{a proof}. }

\newif\iflongform
\longformtrue

\iflongform

\else

\fi

\usepackage{pdfpages}
\usepackage{dcolumn}
\usepackage{amssymb}
\newtheorem{lemma}{Lemma}
\newtheorem{cor}{Corrolary}
\newtheorem{rem}{rem}
\newtheorem{claim}{Claim}
\newtheorem{defn}{Definition}
\newtheorem{assumption}{Assumption}
\usepackage{color}
\usepackage{mathtools}
\usepackage{enumitem}

\usepackage[lefttag]{mhequ}
\def \be{\begin{equs}}
\def \ee{\end{equs}}

\newcommand{\FM}[1]{\mathcal{M}(#1)}

\newcommand{\topology}{\mathcal{T}}

\newcommand{\expect}{\mathbb{E}}
\newcommand{\E}{\mathbb{E}}
\newcommand{\Prob}{\mathbb{P}}
\newcommand{\IProb}{\mathbb{Q}}
\newcommand{\kernel}{\{P_{x}(\cdot)\}_{x \in X}}
\newcommand{\kerp}{\{G_{i}(\cdot)\}_{i \in S}}
\newcommand{\asfc}{C t_{L}}

\newcommand{\lak}{\mathcal{T}_{L}}
\newcommand{\trk}{\mathcal{T}^{(S)}}
\newcommand{\gkernel}{\{g(x,1,\cdot)\}_{x\in X}}
\newcommand{\cP}{\mathcal{P}}

\begin{document}
\begin{frontmatter}
\title{Mixing Times and Hitting Times for General Markov Processes}
\runtitle{Mixing and Hitting Times}

\begin{aug}
  \author{\fnms{Robert M.}  \snm{ Anderson}\corref{}\ead[label=e1]{anderson@econ.berkeley.edu }},
  \author{\fnms{Duamu} \snm{Haosui}\ead[label=e2]{duanmuhaosui@berkeley.edu}}
  \and
  \author{\fnms{Aaron}  \snm{Smith}%
  \ead[label=e3]{asmi28@uottawa.ca}%
  \ead[label=u1,url]{http://www.foo.com}}

  \runauthor{R. Anderson et al.}

  \affiliation{University of California, Berkeley and University of Ottawa}

  \address{\printead{e1}}
  \address{\printead{e2}}
  \address{\printead{e3}}

\end{aug}

\begin{abstract}
The hitting and mixing times are two fundamental quantities associated with Markov chains. In \citet{finitemixhit, oliveira2012mixing}, the authors show that the mixing times and ``worst-case" hitting times of reversible Markov chains on finite state spaces are equal up to some universal multiplicative constant. We use tools from nonstandard analysis to extend this result to reversible Markov chains on general state spaces that satisfy the strong Feller property. We also show how this asymptotic equivalence can be used to find bounds on the mixing times of a large class of Markov chains used in MCMC, such as typical Gibbs samplers and Metroplis-Hastings chains, even though they usually do not satisfy the strong Feller property. 
\end{abstract}


\end{frontmatter}

\section{Introduction}
Two of the most-studied quantities in the Markov chain literature are the \textit{mixing time} and \textit{hitting times} associated with a chain. In \citep{finitemixhit, oliveira2012mixing}, the authors showed that these quantities are equal up to universal constants for reversible discrete time Markov processes with finite state space. In this paper, we use Nonstandard Analysis to extend this result to discrete time Markov processes on $\sigma$-compact state spaces that satisfy the \textit{strong Feller property} (see \ref{mixhit}). As in the context of \citep{finitemixhit, oliveira2012mixing}, it is generally easier to get upper bounds on hitting times, and it is generally easier to get lower bounds on mixing times. These results let us estimate whichever is more convenient.

Recall that the mixing time measures the number of steps a Markov chain must take to (approximately) forget its initial condition. This quantity is fundamental in computer science and computational statistics, where it measures the efficiency of algorithms based on Markov chains; it is also important in the statistical physics literature, where it provides a way to qualitatively describe the behaviour of a material (see \textit{e.g.} overviews \citep{markovmix, diaconis2009markov,montenegro2006mathematical,aldous2002reversible,meyn2012markov}). 
The hitting time measures the number of steps a Markov chain must take before entering a set for the first time. This quantity is not always directly relevant for applications, but it is usually easier to compute or estimate and many tools have been developed for estimating hitting times and relating them to other quantities of interest (see \textit{e.g.} the role of hitting time calculations in the theory of metastability \citep{bovier2006metastability}).

\subsection{Nonstandard Analysis} 
In this paper, we extend known results about Markov Processes with a finite state space to those with a continuum state space. Our arguments are based on nonstandard analysis, which allows construction of a single object---a hyperfinite probability space---that satisfies all the first order logic properties of a finite probability space, but can simultaneously be viewed via the Loeb construction (\citep{Loeb75}) as a measure-theoretic probability space. This construction often allows one to make discrete arguments about the hyperfinite probability space, and then use the Loeb construction to express the results in measure-theoretic terms. 

In order to do this, one has to establish approriate notions of liftings (hyperfinite processes that sit ``above'' the measure-theoretic objects of interest) and pushdowns (projections of hyperfinite objects to the measure-theoretic objects).  These liftings and pushdowns form a ``dictionary'' that must be chosen specifically to represent the type of probabilistic process of interest.  Dictionaries for Lebesgue Integration, Brownian Motion and It\^o Integration were given in \citep{andersonisrael} and \citep{anderson87}, for stochastic differential equations in \citep{Keisler87}, and for Markov
chains in \citep{Markovpaper}.   One of the main contributions of this paper is an expansion of the dictionary for Markov chains in \citep{Markovpaper}. This expansion lets us translate the proofs of existing discrete results to obtain several new results, and we anticipate it being useful for the translation of further Markov chain results in the future.

\subsection{Related Literature}  \label{SecRelLit}

\subsubsection{Computational Statistics}

Although Markov chains on infinite state spaces occur in many areas, we are especially interested in the mixing properties of Markov chains used in Markov chain Monte Carlo (MCMC) algorithms. Most algorithms used in MCMC do \textit{not} satisfy the strong Feller condition, and so our main result, Theorem \ref{mixhit}, does not apply directly. In Section \ref{statapp}, we explain how our main result can still be applied to popular MCMC chains.

We note that most chains used in MCMC are \textit{geometrically ergodic} but do not have finite mixing times. Our main results can still be applied in this situation, and this is the subject of a companion paper.

\subsubsection{Equivalence and Sensitivity}

There are many different ways to measure the time it takes for a Markov chain to ``get random." The present paper belongs to the large literature, started in \citep{aldous1982some,aldous1997mixing}, devoted to understanding how much different measures of this time can disagree. 

These ``equivalence" results are closely related to the problem of studying the \textit{sensitivity} of Markov chains to qualitatively-small changes (see \textit{e.g.} \citep{addario2017mixing,hermon2018sensitivity}) and to the study of \textit{perturbations} of Markov chains (see \textit{e.g.} \citep{mitrophanov2005sensitivity,herve2014approximating,pillai2014ergodicity,rudolf2018perturbation,bardenet2017markov,negrea2017error}).  While perturbations have been studied on very general state space, to our knowledge all research related to sensitivity has been focused on Markov chains on discrete state spaces.

Finally, the relationship between hitting and mixing times has been refined since \citep{finitemixhit, oliveira2012mixing}; see \textit{e.g.} \citep{basu2015characterization}.

\subsection{Overview of the Paper}
In Section \ref{secpre}, we give basic definitions and inequalities related to mixing and hitting times. We also state the main results. 

In Section \ref{sechyprob} and Section \ref{sechymarkov}, we introduce hyperfinite representations for probability spaces and discrete-time Markov processes developed in \citep{Markovpaper}. Namely, we show that, for every discrete-time Markov process satisfying appropriate conditions, there exists a corresponding hyperfinite Markov process. 

In Section \ref{secmha}, we show that the mixing times and hitting times of a discrete-time Markov process on compact state space can be approximated by the mixing times and hitting times of its corresponding hyperfinite Markov process. This leads to a proof in Section \ref{SecMixHitCompact} that mixing times and hitting times are asymptotically equivalent for discrete-time Markov processes on compact state space satisfying Section \ref{assumptiondsf}. We extend to $\sigma$-compact spaces in Section \ref{secsigcomp}. 

Finally, in Section \ref{statapp} and Section \ref{AppOtherExt} we show how to apply our results to some popular  chains from statistics.

Various elementary proofs and lemmas are deferred to the appendices.

\section{Preliminaries and Main Results}\label{secpre}

We fix a $\sigma$-compact metric state space $X$ endowed with Borel $\sigma$-algebra $\BorelSets X$ and let $\{P_{x}(\cdot)\}_{x\in X}$ denote the transition kernel of a Markov process with unique stationary measure $\pi$. Throughout the paper, we include $0$ in $\Nats$. For $x\in X$, $t\in \Nats$ and $A\in \BorelSets X$, we write $P_{x}^{(t)}(A)$ or $P^{(t)}(x,A)$ for the transition probability from $x$ to $A$ in $t$ steps. We write $P_{x}(A)$ and $P(x,A)$ as an abbreviation for $P_{x}^{(1)}(A)$ and $P^{(1)}(x,A)$, respectively. Recall that $\{P_{x}(\cdot)\}_{x\in X}$ is said to be \emph{reversible} if
\[ \label{EqDefRev}
\int_{A}P(x,B)\pi(\dee x)=\int_{B}P(x,A)\pi(\dee x).
\]
for every $A,B\in \BorelSets X$.

For probability measures $\mu, \nu$ on $(X, \BorelSets X)$, we denote by
\[
\parallel \mu - \nu \parallel = \sup_{A \in \BorelSets X} |\mu(A) - \nu(A)|
\]
the usual \textit{total variation distance} between $\mu$ and $\nu$.  Our main result will require the following continuity condition on $\{P_{x}(\cdot)\}_{x\in X}$.

\begin{assumption}{DSF}\label{assumptiondsf}
The transition kernel $\{P_{x}(\cdot)\}_{x\in X}$ satisfies the \emph{strong Feller property} if for every $x\in X$ and every $\epsilon>0$ there exists $ \delta>0$ such that
\[
(\forall y\in X) (|y-x|<\delta \implies ( \parallel P_{x}(\cdot) - P_{y}(\cdot) \parallel <\epsilon)).
\]
\end{assumption}

We define the mixing time:

\begin{defn}\label{defmix}
Let $\epsilon\in \PosReals$.
The mixing time $t_{m}(\epsilon)$ of $\{P_{x}(\cdot)\}_{x\in X}$ is
\[
\min\{t\geq 0: d(t)\leq \epsilon\},
\]
where $d(t)=\sup_{x\in X}\parallel P^{(t)}(x,\cdot)-\pi(\cdot) \parallel$.
\end{defn}

It is clear that $d(t)$ is a non-increasing function. The ``lazy" transition kernel associated with $\{P_{x}(\cdot)\}_{x\in X}$ is:

\begin{defn}\label{deflazy}
The \emph{lazy kernel} $\{P_{L}(x,\cdot)\}_{x\in X}$ of a transition kernel $\{P(x,\cdot)\}_{x\in X}$ is given by
\[
P_{L}(x,A)=\frac{1}{2}P(x,A)+\frac{1}{2}\delta(x,A)
\]
for every $x\in X$ and every $A\in \BorelSets X$, where
$$
\delta(x,A)= \begin{cases}
1, \qquad x\in A \\
0, \qquad x\not\in A.
\end{cases}
$$

\end{defn}

For $\epsilon\in \PosReals$, we denote the mixing time of the lazy chain by $t_{L}(\epsilon)$.  For notational convenience, we will simply write $t_{L}$ and $t_m$ when $\epsilon=\frac{1}{4}$. 

We now denote by $\{X_{t}\}_{t \in \Nats}$ a Markov chain with transition kernel $\{P_{x}(\cdot)\}_{x\in X}$ and arbitrary starting point $X_{0} = x_{0} \in X$. Recall that the \textit{hitting time} of a set $A \in \BorelSets X$ for this Markov chain is defined to be:
\[
\tau_{A} = \min \{ t \in \Nats \, : \, X_{t} \in A \}.
\]

We now introduce the maximum hitting time of large sets.

\begin{defn}
Let $\alpha\in \PosReals$. The maximum hitting time with respect to $\alpha$ is
\[
t_{H}(\alpha)=\sup\{\expect_{x}(\tau_{A}): x\in X, A\in \BorelSets X\ \text{such that}\ \pi(A)\geq \alpha\},
\]
where $\expect$ is the expectation of a measure in the product space which generates the underlying Markov process and its subscript refers to the starting point $X_{0}$.
\end{defn}

We now quote the main results from \citep{finitemixhit, oliveira2012mixing}: 

\begin{thm}[{\citep[][Thm.~1.1]{finitemixhit}};{\citep[][Thm.~1.3]{oliveira2012mixing}}]\label{fmixhit}
Let $0 < \alpha<\frac{1}{2}$.
Then there exist universal positive constants $c'_{\alpha},c_{\alpha}$ so that for every finite reversible Markov process
\[
c'_{\alpha}t_{H}(\alpha)\leq t_{L}\leq c_{\alpha}t_{H}(\alpha).
\]
\end{thm}

Throughout the paper, we denote by $\mathcal{M}$ the collection of discrete time reversible transition kernels with a stationary distribution on a $\sigma$-compact metric state space satisfying Section \ref{assumptiondsf}.
Note that transition kernels on finite state spaces belong to $\mathcal{M}$. The main result of this paper generalizes  Theorem \ref{fmixhit} to $\mathcal{M}$:  

\begin{thm}\label{mixhit}
Let $0 < \alpha<\frac{1}{2}$.
Then there exist universal constants $0<a_{\alpha},a'_{\alpha}<\infty$ such that, for every $\kernel \in \mathcal{M}$, we have
\[
a'_{\alpha}t_{H}(\alpha)\leq t_{L}\leq a_{\alpha}t_{H}(\alpha).
\]
\end{thm}

The \textit{first} inequality in Theorem \ref{mixhit} is straightforward and well-known (see \textit{e.g.} Lemma \ref{maxhitless}). The second is more difficult. The compact version of Theorem \ref{mixhit} is proved in Theorem \ref{mixhitcompact} and the general version is proved in Theorem \ref{mixhitpf}.

\subsection{Equivalent Form of Mixing Times and Hitting times}\label{secequivalent}
In this section, we define several quantities that are asymptotically equivalent to the mixing times and the maximum hitting times defined in the previous section. These equivalent forms play important roles throughout the entire paper, since they are easier to work with for general Markov processes. First, let 
\[
\overline{d}(t)=\sup_{x,y\in X}\parallel P^{(t)}(x,\cdot)-P^{(t)}(y,\cdot) \parallel.
\]

We recall two important results on $\overline{d}(t)$:\footnote{The referenced proofs are stated for discrete spaces, but the arguments apply immediately in the current setting.}

\begin{lemma}[{\citep[][Lemma.~4.11]{markovmix}}]\label{mixequal}
For every $t\in \Nats$, we have $d(t)\leq \overline{d}(t)\leq 2d(t)$.
\end{lemma}

\begin{lemma}[{\citep[][Lemma.~4.12]{markovmix}}]\label{submulti}
The function $\overline{d}$ is sub-multiplicative. That is, for $s, t \in \Nats$,
\[
\overline{d}(s+t)\leq \overline{d}(s)\overline{d}(t).
\]
\end{lemma}

For every $\epsilon\in \PosReals$, define the \emph{standardized mixing time} to be
\[
\overline{t}_{m}(\epsilon)=\min\{t\geq 0: \overline{d}(t)\leq \epsilon\}.
\]

Similarly, we define
\[
\overline{t}_{L}(\epsilon)=\min\{t\geq 0: \overline{d}_{L}(t)\leq \epsilon\}.
\]
For convenience, we write $\overline{t}_{m}$ and $\overline{t}_{L}$ when $\epsilon=\frac{1}{4}$. The following well-known equivalence between mixing times and standard mixing times follows immediately from \ref{mixequal} and \ref{submulti}:

\begin{lemma}\label{mixequivalent}
For every transition kernel $\{P_{x}(\cdot)\}_{x\in X}$, we have $\overline{t}_{m}\leq 2t_{m}\leq 2\overline{t}_{m}$.
\end{lemma}
%

Next, we define the \emph{large hitting time}:

\begin{defn}\label{largehit}
Let $\alpha\in \PosReals$. The large hitting time with respect to $\alpha$ is
\[
\tau_{g}(\alpha)=\min\{t\in \Nats: \inf\{\Prob_{x}(\tau_{A}\leq t): x\in X, A\in \BorelSets X \ \text{such that}\ \pi(A)\geq \alpha\}>0.9\}
\]
where $\Prob$ is a measure in the product space which generates the underlying Markov process and its subscript gives the starting point of the Markov process.
\end{defn}

Unsurprisingly, the maximum hitting time is asymptotically equivalent to the large hitting time:

\begin{lemma}\label{maxlarge}
For every $\alpha\in \PosReals$ and every Markov process,
we have $0.1\tau_{g}(\alpha)\leq t_{H}(\alpha)\leq 2\tau_{g}(\alpha)$.
\end{lemma}

\begin{proof}
See Appendix \ref{SecElEq}.
\end{proof}

The following is an immediate consequence of \ref{maxlarge} and Theorem \ref{fmixhit}. 

\begin{thm}\label{mhequal}
Let $0 < \alpha<\frac{1}{2}$.
Then there exist universal positive constants $c'_{\alpha},c_{\alpha}$ so that for every finite reversible Markov process
\[
c'_{\alpha}\tau_{g}(\alpha)\leq t_{L}\leq c_{\alpha}\tau_{g}(\alpha).
\]
\end{thm}

\subsection{Notation from Nonstandard Analysis}
In this paper, we use nonstandard analysis, a powerful machinery derived from mathematical logic, as our main toolkit. For those who are not familiar with nonstandard analysis, \citep{Markovpaper,nsbayes} provide reviews tailored to statisticians and probabilists. \citep{NSAA97,NDV,NAW} provide thorough introductions. 

We briefly introduce the setting and notation from nonstandard analysis.
We use $\NSE{}$ to denote the nonstandard extension map taking elements, sets, functions, relations, etc., to their nonstandard counterparts.
In particular, $\HReals$ and $\NSE{\Nats}$ denote the nonstandard extensions of the reals and natural numbers, respectively.
An element $r\in \HReals$ is \emph{infinite} if $|r|>n$ for every $n\in \Nats$ and is \emph{finite} otherwise. An element $r \in \HReals$ with $r > 0$ is \textit{infinitesimal} if $r^{-1}$ is infinite. For $r,s \in \HReals$, we use the notation $r \approx s$ as shorthand for the statement ``$|r-s|$ is infinitesimal," and similarly we use use $r \gtrapprox s$ as shorthand for the statement ``either $r \geq s$ or $r \approx s$."

Given a topological space $(X,\topology)$,
the monad of a point $x\in X$ is the set $\bigcap_{ U\in \topology \, : \, x \in U}\NSE{U}$.
An element $x\in \NSE{X}$ is \emph{near-standard} if it is in the monad of some $y\in X$.
We say $y$ is the standard part of $x$ and write $y=\ST(x)$. Note that such $y$ is unique.
We use $\NS{\NSE{X}}$ to denote the collection of near-standard elements of $\NSE{X}$
and we say $\NS{\NSE{X}}$ is the \emph{near-standard part} of $\NSE{X}$.
The standard part map $\ST$ is a function from $\NS{\NSE{X}}$ to $X$, taking near-standard elements to their standard parts.
In both cases, the notation elides the underlying space $Y$ and the topology $T$,
because the space and topology will always be clear from context.
For a metric space $(X,d)$, two elements $x,y\in \NSE{X}$ are \emph{infinitely close} if $\NSE{d}(x,y)\approx 0$.
An element $x\in \NSE{X}$ is near-standard if and only if it is infinitely close to some $y\in X$.
An element $x\in \NSE{X}$ is finite if there exists $y\in X$ such that $\NSE{d}(x,y)<\infty$ and is infinite otherwise.

Let $X$ be a topological space endowed with Borel $\sigma$-algebra $\BorelSets X$ and
let $\FM{X}$ denote the collection of all finitely additive probability measures on $(X,\BorelSets X)$.
An internal probability measure $\mu$ on $(\NSE{X},\NSE{\BorelSets X})$ is an element of $\NSE{\FM{X}}$.
Namely, an internal probability measure $\mu$ on $(\NSE{X},\NSE{\BorelSets X})$
is an internal function from $\NSE{\BorelSets X}\to \NSE{[0,1]}$ such that
\begin{enumerate}
\item $\mu(\emptyset)=0$;
\item $\mu(\NSE{X})=1$; and
\item $\mu$ is hyperfinitely additive.
\end{enumerate}
The Loeb space of the internal probability space $(\NSE{X},\NSE{\BorelSets X}, \mu)$ is a countably additive probability space $(\NSE{X},\Loeb{\NSE{\BorelSets X}}, \Loeb{\mu})$ such that
\[
\Loeb{\NSE{\BorelSets X}}=\{A\subset \NSE{X}|(\forall \epsilon>0)(\exists A_i,A_o\in \NSE{\BorelSets X})(A_i\subset A\subset A_o\wedge \mu(A_o\setminus A_i)<\epsilon)\}
\]
and
\[
\Loeb{\mu}(A)=\sup\{\ST(\mu(A_i))|A_i\subset A,A_i\in \NSE{\BorelSets X}\}=\inf\{\ST(\mu(A_o))|A_o\supset A,A_o\in \NSE{\BorelSets X}\}.
\]

Every standard model is closely connected to its nonstandard extension via the \emph{transfer principle}, which asserts that a first order statement is true in the standard model is true if and only if it is true in the nonstandard model.
Finally, given a cardinal number $\kappa$, a nonstandard model is called $\kappa$-saturated if the following condition holds:
let $\cF$ be a family of internal sets, if $\cF$ has cardinality less than $\kappa$ and $\cF$ has the finite intersection property, then the total intersection of $\cF$ is non-empty. In this paper, we assume our nonstandard model is as saturated as we need (see \textit{e.g.} {\citep[][Thm.~1.7.3]{NSAA97}} for the existence of $\kappa$-saturated nonstandard models for any uncountable cardinal $\kappa$).

\section{Hyperfinite Representation of Probability Spaces}\label{sechyprob}

In this section, we give an overview of hyperfinite representation for probability spaces developed in \citep{Markovpaper}.
All the proofs can be found in {\citep[][Section.~6]{Markovpaper}}.
We use similar notation to \citep{Markovpaper} and \citep{nsbayes}.
The following theorem gives a nonstandard characterization for compact topological spaces. 

\begin{thm} [{\citep[][Thm.~4.1.13]{AR65}}]\label{HBfinite}
A topological space $X$ is compact if and only if every $x\in \NSE{X}$ is near-standard. 
\end{thm}

In the following, we use the common notation $d(x,A) = \inf \{y \in X \, : \, d(x,y) \}$ for every $x\in X$ and every $A\subset X$.

We now introduce the concept of hyperfinite representation of a Heine-Borel metric space $X$. The intuition is to take a ``large enough" portion of $\NSE{X}$ containing $X$ and then partition it into hyperfinitely many pieces of *Borel sets with infinitesimal radius. We then pick one ``representative" from each piece to form a hyperfinite set. The formal definition is given below.

\begin{defn}\label{hyperapproxsp}
Let $(X,d)$ be a metric space satisfying the Heine-Borel condition. Let $\delta\in {^{*}\Reals^{+}}$ be an infinitesimal and $r$ be an infinite hyperreal number. A $(\delta,r)$-hyperfinite representation of $X$ is a tuple $(S,\{B(s)\}_{s\in S})$ such that

\begin{enumerate}
\item $S$ is a hyperfinite subset of $\NSE{X}$.
\item $s\in B(s)\in \NSE{\BorelSets X}$ for every $s\in S$.
\item For every $s\in S$, the diameter of $B(s)$ is no greater than $\delta$.
\item $B(s_1)\cap B(s_2)=\emptyset$ for every $s_1\neq s_2\in S$.
\item For any $x\in \NS{^{*}X}$, ${^{*}d}(x,{^{*}X}\setminus \bigcup_{s\in S}B(s))>r$.
\item There exists $a_0\in X$ and some infinite $r_0$ such that
\[
\NS{^{*}X}\subset \bigcup_{s\in S}B(s)=\overline{U(a_0,r_0)}
\]
where $\overline{U(a_0,r_0)}=\{x\in {^{*}X}: {^{*}d}(x,a_0)\leq r_0\}$.
\end{enumerate}
The set $S$ is called the \emph{base set} of the hyperfinite representation. For every $x\in \bigcup_{s\in S}B(s)$, we use $s_{x}$ to denote the unique element in $S$ such that $x\in B(s_x)$.
\end{defn}

If $X$ is compact, we have $\NS{\NSE{X}}=\NSE{X}$ by \ref{HBfinite}. In this case, we can pick $S$ such that $\bigcup_{s\in S}B(s)=\NSE{X}$, and hence the second parameter of an $(\epsilon,r)$-hyperfinite representation is redundant. Thus, we shall simply work with an $\epsilon$-hyperfinite representation in the case where $X$ is compact.
The set $\overline{U(a_0,r_0)}$ can be seen as the *closure of the nonstandard open ball $U(a_0,r_0)$. As $X$ satisfies the Heine-Borel condition, by the transfer principle, $\overline{U(a_0,r_0)}$ is a *compact set. That is, $\overline{U(a_0,r_0)}$ satisfies all the first-order logic properties of a compact set.

The next theorem shows that hyperfinite representations always exist. 
Although the statement appears to be slightly stronger than {\citep[][Thm.~6.6]{Markovpaper}}, 
its proof is almost identical to the proof of {\citep[][Thm.~6.6]{Markovpaper}} hence is omitted. 

\begin{thm}\label{exhyper}
Let $X$ be a metric space satisfying the Heine-Borel condition. Then, for every positive infinitesimal $\delta$ and every positive infinite $r$, there exists an $(\delta,r)$-hyperfinite representation $(S_{\delta}^{r},\{B(s)\}_{s\in S_{\delta}^{r}})$ of ${^{*}X}$ such that $X\subset S_{\delta}^{r}$.
\end{thm}

Suppose $X$ is a Heine-Borel metric space endowed with Borel $\sigma$-algebra $\BorelSets X$. Let $P$ be a probability measure on $(X,\BorelSets X)$. Let $S$ be the base set of a $(\delta,r)$-hyperfinite representation of $X$ for some positive infinitesimal $\delta$ and some positive infinite number $r$. The next theorem shows that we can define an internal probability measure on $(S,\mathcal{I}(S))$ that gives a ``nice" approximation of $P$.

\begin{thm}[{\citep[][Thm.~6.11]{Markovpaper}}]\label{representationthm}
Let $(X,\BorelSets X,P)$ be a Borel probability space
where $X$ is a metric space satisfying the Heine-Borel condition,
and let $(\NSE{X},{}^{*}\BorelSets X,{^{*}P})$ be its nonstandard extension.
For every positive infinitesimal $\delta$, every positive infinite $r$ and every $(\delta,r)$-hyperfinite representation $(S,\{B(s)\}_{s\in S})$ of ${^{*}X}$, define an internal probability measure $P'$ on $(S,\mathcal{I}(S))$ by letting $P'(\{s\})=\frac{\NSE{P}(B(s))}{{^{*}P(\bigcup_{t\in S}B(t))}}$ for every $s\in S$.
Then we have

\begin{enumerate}
\item $P'(\{s\})\approx {^{*}P}(B(s))$.

\item ${^{*}P(\bigcup_{s\in S}B(s))}\approx 1$.

\item $P(E)=\Loeb{P'}(\ST^{-1}(E)\cap S)$ for every $E\in \BorelSets X$.
\end{enumerate}
\end{thm}

\section{Hyperfinite Representation of General Markov Processes}\label{sechymarkov}

Let $\{ P_{x} \}_{x \in X}$ be the transition kernel of a discrete-time Markov process with state space $X$. We assume that $X$ is a metric space satisfying the Heine-Borel condition throughout the rest of the paper until otherwise mentioned. The transition probability can be viewed as a function $g: X\times \Nats\times \BorelSets X\to [0,1]$ by letting $g(x,t,A)=P_{x}^{(t)}(A)$ for every $x\in X$, $t\in \Nats$ and $A\in \BorelSets X$. We will use $g(x,t,A)$ and $P_{x}^{(t)}(A)$ interchangeably. For any $x\in X$ and $A\in \BorelSets X$, let $P_{x}^{(0)}(A)=1$ if $x\in A$ and $P_{x}^{(0)}(A)=0$ if $x\not\in A$.
We will construct a hyperfinite object to represent the Markov process $\{X_t\}_{t\in \Nats}$ associated with the transition kernel $g$. We fix a set $T=\{1,2,\dotsc,K\}$ for some infinite $K\in \NSE{\Nats}$ throughout the paper. A \emph{hyperfinite Markov process} is defined analogously to finite Markov processes. Namely, a hyperfinite Markov process is characterized by the following four ingredients:
\begin{enumerate}
\item A \emph{state space} $S$ which is a non-empty hyperfinite set.
\item A time line $T$.
\item A set $\{\nu_i: i\in S\}\subset \NSE{\Reals}$ where each $\nu_i\geq 0$ and $\sum_{i\in S}\nu_i=1$.
\item A set $\{p_{ij}\}_{i,j\in S}$ of non-negative hyperreals with $\sum_{j\in S}p_{ij}=1$ for every $i\in S$.
\end{enumerate}

The following theorem shows that it is always possible to construct a hyperfinite Markov processes with these parameters.

\begin{thm}[{\citep[][Thm.~7.2]{Markovpaper}}]\label{HMexist}
Fix a hyperfinite state space $S$, a time line $T$, a hyperfinite set $\{v_i\}_{i\in S}$ and a hyperfinite set $\{p_{ij}\}_{i,j\in S}$ that satisfy the immediately-preceding conditions.
Then there exists an internal probability triple $(\Omega,\cA,\Prob)$ with an internal stochastic process $\{X_t\}_{t\in T}$ defined on $(\Omega,\cA,\Prob)$ such that
\[
\Prob(X_0=i_0,X_{\delta t}=i_{\delta t},...X_{t}=i_t)=v_{i_{0}}p_{i_{0}i_{\delta t}}...p_{i_{t-\delta t}i_{t}}
\]
for all $t\in T$ and $i_0,....i_{t}\in S$.
\end{thm}

As in \citep{Markovpaper}, we will construct a hyperfinite Markov process $\{X'_t\}_{t\in T}$ which is a ``nice" representation of $\{X_t\}_{t\in \Nats}$. Due to the similarities between finite objects and hyperfinite objects, $\{X'_t\}_{t\in T}$ inherits many key properties from finite Markov processes. $\{X'_t\}_{t\in T}$ will play an essential role throughout the paper.

Pick any positive infinitesimal $\delta$ and any positive infinite number $r$.
Let $(S,\{B(s)\}_{s\in S})$ be a $(\delta, r)$-hyperfinite representation of $\NSE{X}$.
Let us recall some key properties of $(S,\{B(s)\}_{s\in S})$.
\begin{enumerate}
\item $s\in B(s)$ for every $s\in S$.
\item For every $s\in S$, the diameter of $B(s)$ is no greater than $\delta$.
\item $B(s_1)\cap B(s_2)=\emptyset$ for every $s_1\neq s_2\in S$.
\item $\NS{^{*}X}\subset \bigcup_{s\in S}B(s)$.
\end{enumerate}
For every $x\in {^{*}X}$, we know that ${^{*}g}(x,1,.)$ is an internal probability measure on $({^{*}X}, {^{*}\BorelSets X})$. We can construct $(S, \{B(s)\}_{s\in S})$ so that:

\begin{lemma}[{\citep[][Lemma.~9.14]{Markovpaper}}]\label{dsfconsequence}
Suppose that $g$ satisfies Section \ref{assumptiondsf}. 
There exists a hyperfinite representation $(S,\{B(s)\}_{s\in S})$ of $\NSE{X}$ such that, for every $s\in S$, every positive $n\in \Nats$ and every $A\in \NSE{\BorelSets X}$, we have
\[
\NSE{g}(x_1,n,A)\approx \NSE{g}(x_2,n,A)
\]
for every $x_1,x_2\in B(s)$.
\end{lemma}

We shall fix such a $(S, \{B(s)\}_{s\in S})$ for the rest of the paper. 
When $X$ is non-compact, $\bigcup_{s\in S}B(s)\neq {^{*}X}$. Hence, we need to truncate ${^{*}g}$ to be an internal probability measure on $\bigcup_{s\in S}B(s)$.

\begin{defn}\label{dtrunchain}
For $i\in \{0,1\}$, let $g'(x,i,A): \bigcup_{s\in S}B(s)\times {^{*}\BorelSets X} \to {^{*}[0,1]}$ be given by:
\[
g'(x,i,A)={^{*}g(x,i,A\cap\bigcup_{s\in S}B(s))}+\delta(x,A){^{*}g(x,i, {^{*}X}\setminus\bigcup_{s\in S}B(s))}.
\]
where $\delta(x,A)=1$ if $x\in A$ and $\delta(x,A)=0$ if otherwise.
\end{defn}

We now define a hyperfinite Markov process $\{X'_t\}_{t\in T}$ on $S$ by specifying its internal transition kernels.
We will use $G_{i}^{(t)}(\{j\})$ or $G_{ij}^{(t)}$ to denote the internal transition probability from $i$ to $j$ at time $t$.
For $i,j\in S$, define $G_{ij}^{(0)}=g'(i,0,B(j))$ and $G_{ij}=g'(i,1,B(j))$.
For every $t\in T$, we define $G_{ij}^{(t)}$ by the inductive formula $G_{ij}^{(t+1)} = \sum_{k} G_{ik} G_{kj}^{(t)}$.
For any internal set $A\subset S$ and any $i\in S$, let $G_{i}^{(0)}(A)=\sum_{j\in A}G_{ij}^{(0)}$ and $G_{i}(A)=\sum_{j\in A}G_{ij}$.
It follows from definition that $G_{ij}^{(0)}=1$ if $i=j$ and $G_{ij}^{(0)}=0$ otherwise.
It is straightforward to verify that $G_{i}^{(t)}(\cdot)$ defines an internal probability measure on $S$ for every $t\in \Nats$ and $i\in S$.

\begin{lemma}[{\citep[][Lemma.~8.13]{Markovpaper}}]\label{allinternal}
For any $i\in S$ and any $t\in \Nats$, $G_{i}^{(t)}(\cdot)$ is an internal probability measure on $(S,\mathcal{I}(S))$.
\end{lemma}

We now quote the following two key results from \citep{Markovpaper}.
\begin{thm}[{\citep[][Thm.~8.14]{Markovpaper}}]\label{Gapprox}
Suppose $\{g(x,1,\cdot)\}_{x\in X}$ satisfies Section \ref{assumptiondsf}.
Then for any $t\in \Nats$, any $s\in \NS{S}$ and any $A\in {^{*}\BorelSets X}$, we have ${^{*}g}(s,t,\bigcup_{a\in A\cap S}B(a))\approx G_{s}^{(t)}(A\cap S)$.
\end{thm}

\begin{thm}[{\citep[][Lemma.~8.15]{Markovpaper}}]\label{disrepresent}
Suppose $\{g(x,1,\cdot)\}_{x\in X}$ satisfies Section \ref{assumptiondsf}. Then for any $s\in \NS{S}$, any $t\in \Nats$ and any $E\in \BorelSets X$, $g(\ST(s),t,E)=\overline{G}_{s}^{(t)}(\ST^{-1}(E)\cap S)$.
\end{thm}

These theorems shows that the transition probabilities of $\{X_t\}_{t\in \Nats}$ agree with the Loeb extension of the internal transition probabilities of $\{X'_t\}_{t\in T}$ via standard part map. Such $\{X'_t\}_{t\in \Nats}$ is called a hyperfinite representation of $\{X_t\}_{t\in \Nats}$.

\subsection{Hyperfinite Representation of Lazy Chain}\label{seclazy}
For discrete-time Markov processes, one considers a lazy version of the original Markov process to avoid periodicity and near-periodicity issues. Let $g$ be the transition kernel of a discrete-time Markov process. We denote by $g_{L}$ the lazy kernel of $g$, given by the formula $g_L(x,1,A)=\frac{1}{2}g(x,1,A)+\frac{1}{2}\delta(x,A)$ for every $x\in X$ and every $A\in \BorelSets X$, where we recall $\delta(x,A)=1$ if $x\in A$ and $\delta_{x}(A)=0$ if $x\not\in A$.
Note that $g_{L}$ generally does not satisfy Section \ref{assumptiondsf} even if $g$ does. Suppose $g$ satisfies Section \ref{assumptiondsf} and let $G$ be a hyperfinite representation of $g$. In addition, let $\{X'_t\}_{t\in T}$ be a hyperfinite Markov process associated with the internal transition kernel $G$. Both $G$ and $\{X'_t\}_{t\in T}$ will be fixed until the applications in Section \ref{statapp}.

The lazy chain of $\{X'_t\}_{t\in T}$ is defined to be a hyperfinite Markov process with transition probabilities $L_{ij}^{(0)}=G_{ij}^{(0)}$ and $L_{ij}=\frac{1}{2}G_{ij}+\frac{1}{2}\Delta(i,j)$, where $\Delta(i,j)=1$ if $i=j$ and $\Delta(i,j)=0$ if $i\neq j$. Thus, for every $i\in S$ and $A\in \mathcal{I}(S)$ we have
\[
L_{i}(A)=\sum_{j\in A}L_{ij}=\sum_{j\in A}(\frac{1}{2}G_{ij}+\frac{1}{2}\Delta(i,j))=\frac{1}{2}G_{i}(A)+\frac{1}{2}\Delta(i,A)
\]
where $\Delta(i,A)=1$ if $i\in A$ and $\Delta(i,A)=0$ if $i\not\in A$. For every $i\in S$, $A\in \mathcal{I}(S)$ and every $t\in T$, we define $L_{i}^{(t+1)}(A)$ by the inductive formula $L_{i}^{(t+1)}(A)=\sum_{j\in S}L_{ij}L_{j}^{(t)}(A)$.

Before proving the main result of this section, we quote the following useful lemma.
\begin{lemma}[{\citep[][Lemma.~7.24]{Markovpaper}}] \label{tvfunction}
Let $P_1$ and $P_2$ be two internal probability measures on a hyperfinite set $S$. Then
\[
\parallel P_1(\cdot)-P_2(\cdot) \parallel \geq \NSE{\sup}_{f: S\to {^{*}[0,1]}}|\sum_{i\in S}P_{1}(\{i\})f(i)-\sum_{i\in S}P_{2}(\{i\})f(i)|,
\]
where $\parallel P_1(\cdot)-P_2(\cdot) \parallel=\NSE{\sup}_{A\in \mathcal{I}(S)}|P_{1}(A)-P_{2}(A)|$ and the $\NSE{\sup}$ is taken over all internal functions.
\end{lemma}

We now prove the following representation theorem, which is similar in spirit to Theorem \ref{Gapprox}:

\begin{thm}\label{lazystar}
Suppose $\{g(x,1,\cdot)\}_{x\in X}$ satisfies Section \ref{assumptiondsf}.
Then for any $t\in \Nats$, any $x\in \NS{\NSE{X}}$ and any $A\in {^{*}\BorelSets X}$,
we have ${^{*}g_{L}}(x,t,\bigcup_{a\in A\cap S}B(a))\approx L_{s_{x}}^{(t)}(A\cap S)$ where $s_{x}$ is the unique element in $S$ such that $x\in B(s_x)$.
\end{thm}
\begin{proof}
We prove it by induction on $t\in \Nats$. Let $t=1$. Pick any $x\in \NS{\NSE{X}}$ and any $A\in \NSE{\BorelSets X}$, by Lemma \ref{dsfconsequence} and Theorem \ref{Gapprox}, we have
\[
{^{*}g_{L}}(x,1,\bigcup_{a\in A\cap S}B(a))&=\frac{1}{2}\NSE{g}(x,1,\bigcup_{a\in A\cap S}B(a))+\frac{1}{2}\NSE{\delta}(x,\bigcup_{a\in A\cap S}B(a))\\
&\approx \frac{1}{2}\NSE{g}(s_x,1,\bigcup_{a\in A\cap S}B(a))+\frac{1}{2}\Delta(s_x,A\cap S)\\
&\approx \frac{1}{2}G_{s_{x}}(A\cap S)+\frac{1}{2}\Delta(s_{x},A\cap S)=L_{s_{x}}(A\cap S).
\]

Suppose the theorem holds for $t=n$. We now show that the theorem holds for $t=n+1$.
By the transfer of the Markov property, we have
\[\label{glcalstart}
&{^{*}g_{L}}(x,n+1,\bigcup_{a\in A\cap S}B(a))\\
&=\int \NSE{g_{L}}(y,n,\bigcup_{a\in A\cap S}B(a))\NSE{g_{L}}(x,1,\dee y)\\
&\approx \int_{\bigcup_{s\in S}B(s)}\NSE{g_{L}}(y,n,\bigcup_{a\in A\cap S}B(a))\NSE{g_{L}}(x,1,\dee y),
\]
where the last $\approx$ follows from the fact that $\NSE{g_{L}}(x,1,\bigcup_{s\in S}B(s))=1$.
By the induction hypothesis, we know that $\NSE{g_{L}}(y,n,\bigcup_{a\in A\cap S}B(a))\approx L_{s_{y}}^{(n)}(A\cap S)$ for every $y\in \bigcup_{s\in S}B(s)$.
Thus, we have
\[
&\int_{\bigcup_{s\in S}B(s)}\NSE{g_{L}}(y,n,\bigcup_{a\in A\cap S}B(a))\NSE{g_{L}}(x,1,\dee y)\\
&\approx \int_{\bigcup_{s\in S}B(s)}L_{s_{y}}^{(n)}(A\cap S)\NSE{g_{L}}(x,1,\dee y)\\
&=\sum_{s\in S}L_{s}^{(n)}(A\cap S)\NSE{g_{L}}(x,1,B(s))\\
&=\sum_{s\in S}L_{s}^{(n)}(A\cap S)(\frac{1}{2}\NSE{g}(x,1,B(s))+\frac{1}{2}\NSE{\delta}(x,B(s)))\\
&=L_{s_x}^{(n)}(A\cap S)(\frac{1}{2}\NSE{g}(x,1,B(s_x))+\frac{1}{2})+\frac{1}{2}\sum_{s\neq s_x}L_{s}^{(n)}(A\cap S)\NSE{g}(x,1,B(s)). \label{glcalend}
\]
We must now calculate the second term. By Lemma \ref{dsfconsequence}, we have $\NSE{g}(x,1,B(s_x))\approx \NSE{g}(s_x,1,B(s_x))$. By Definition \ref{dtrunchain}, we know that $\NSE{g}(s_x,1,B(s_x))\approx G_{s_{x}s_{x}}$. 

We will now show that
\[
\sum_{s\neq s_x}L_{s}^{(n)}(A\cap S)\NSE{g}(x,1,B(s))\approx \sum_{s\neq s_x}L_{s}^{(n)}(A\cap S)\NSE{g}(s_x,1,B(s)) 
\]
by considering two cases: $\NSE{g}(x,1,\bigcup_{s\neq s_x}B(s))\approx 0$ and $\NSE{g}(x,1,\bigcup_{s\neq s_x}B(s))\not\approx 0$. If $\NSE{g}(x,1,\bigcup_{s\neq s_x}B(s))\approx 0$, by Section \ref{assumptiondsf}, we have $\NSE{g}(s_x,1,\bigcup_{s\neq s_x}B(s))\approx 0$.
Thus, we have $\sum_{s\neq s_x}L_{s}^{(n)}(A\cap S)\NSE{g}(x,1,B(s))\approx \sum_{s\neq s_x}L_{s}^{(n)}(A\cap S)\NSE{g}(s_x,1,B(s))\approx 0$.
 
In the case $\NSE{g}(x,1,\bigcup_{s\neq s_x}B(s))\not\approx 0$, by Section \ref{assumptiondsf}, we have $\NSE{g}(s_x,1,\bigcup_{s\neq s_x}B(s))\not\approx 0$.
This allows us to define $P_1,P_2: \mathcal{I}(S)\to \NSE{[0,1]}$ by the formulae $P_1(A)=\frac{\NSE{g}(x,1,\bigcup_{s\in (A\cap S\setminus\{s_x\})}B(s))}{\NSE{g}(x,1,\bigcup_{s\neq s_x}B(s))}$
and $P_2(A)=\frac{\NSE{g}(s_x,1,\bigcup_{s\in (A\cap S\setminus\{s_x\})}B(s))}{\NSE{g}(s_x,1,\bigcup_{s\neq s_x}B(s))}$. Then both $P_1$ and $P_2$ are internal probability measures on $S$.
By Section \ref{assumptiondsf}, we know that $\parallel P_1(\cdot)-P_2(\cdot) \parallel\approx 0$.
By Lemma \ref{tvfunction}, this implies
\[
\sum_{s\neq s_x}L_{s}^{(n)}(A\cap S)\NSE{g}(x,1,B(s))\approx \sum_{s\neq s_x}L_{s}^{(n)}(A\cap S)\NSE{g}(s_x,1,B(s))
\]
in our second case as well, so this equality always holds.

By Definition \ref{dtrunchain}, we know that $\NSE{g}(s_x,1,B(s))=G_{s_{x}s}$ for $s\neq s_{x}$. Hence we always have $\frac{1}{2}\sum_{s\neq s_x}L_{s}^{(n)}(A\cap S)\NSE{g}(x,1,B(s))\approx \frac{1}{2}\sum_{s\neq s_x}L_{s}^{(n)}(A\cap S)G_{s_{x}s}$. Thus, combining \eqref{glcalstart} to \eqref{glcalend}, we have
\[
&{^{*}g_{L}}(x,n+1,\bigcup_{a\in A\cap S}B(a))\\
&=L_{s_x}^{(n)}(A\cap S)(\frac{1}{2}\NSE{g}(x,1,B(s_x))+\frac{1}{2})+\frac{1}{2}\sum_{s\neq s_x}L_{s}^{(n)}(A\cap S)\NSE{g}(x,1,B(s))\\
&\approx L_{s_x}^{(n)}(A\cap S)(\frac{1}{2}G_{s_{x}s_{x}}+\frac{1}{2})+\frac{1}{2}\sum_{s\neq s_x}L_{s}^{(n)}(A\cap S)G_{s_{x}s}.
\]

On the other hand, we have
\[
L_{s_x}^{(n+1)}(A\cap S)&=\sum_{s\in S}L_{s_{x}s}L_{s}^{(n)}(A\cap S)\\
&=\sum_{s\in S}(\frac{1}{2}G_{s_{x}s}+\frac{1}{2}\Delta(s_x,s))L_{s}^{(n)}(A\cap S)\\
&=\sum_{s\neq s_x}\frac{1}{2}G_{s_{x}s}L_{s}^{(n)}(A\cap S)+(\frac{1}{2}G_{s_{x}s_{x}}+\frac{1}{2})L_{s_x}^{(n)}(A\cap S).
\]
Thus, we can conclude that ${^{*}g_{L}}(x,n+1,\bigcup_{a\in A\cap S}B(a))\approx L_{s_{x}}^{(n+1)}(A\cap S)$. By induction, we have the desired result.
\end{proof}

The following well-known nonstandard representation theorem is due to Robert Anderson.
\begin{lemma} [{\citep[][Thm ~3.3]{anderson87}}]\label{loebapproximate}
Let $(X,\BorelSets X,\mu)$ be a $\sigma$-compact Borel probability space.
Then $\ST$ is measure preserving from $({^{*}X},\overline{^{*}\BorelSets X},\overline{^{*}\mu})$ to $(X,\BorelSets X,\mu)$.
That is, we have $\mu(E)=\overline{^{*}\mu}(\ST^{-1}(E))$ for all $E\in \BorelSets X$.
\end{lemma}

Note that every Heine-Borel space is $\sigma$-compact.
We also recall that the hyperfinite state space $S$ of $\{X'_t\}_{t\in T}$ contains $X$ as a subset.
We now present the following hyperfinite representation theorem for lazy chains. 
The proof is very similar to the proof of \ref{disrepresent} hence is omitted. 

\begin{thm}\label{lazystandard}
Suppose that the transition kernel of $\{X_t\}_{t\in \Nats}$ satisfies Section \ref{assumptiondsf}.
Then for every $x\in X$, every $t\in \Nats$ and every $E\in \BorelSets X$, we have $g_{L}(x,t,E)=\Loeb{L}_{x}^{(t)}(\ST^{-1}(E)\cap S)$.
\end{thm}

%
\subsection{Hyperfinite Representation of Stationary Distribution}
Let $\pi$ be a stationary distribution of $\kernel$. We construct an analogous object in the hyperfinite representation $\kerp$, called the ``weakly stationary distribution".

\begin{defn}\label{defnstationary}
Let $\Pi$ be an internal probability measure on $(S,\mathcal{I}(S))$. We say $\Pi$ is a weakly stationary distribution for $\kerp$ if there exists an infinite $t_0\in T$ such that for any $t\leq t_0$ and any $A\in \mathcal{I}(S)$ we have $\Pi(A)\approx\sum_{i\in S}\Pi(\{i\})G_{i}^{(t)}(A)$.
\end{defn}

In \citep{Markovpaper}, the authors show that weakly stationary distributions exist for hyperfinite representations of general state space continuous time Markov processes under moderate regularity conditions. In this section, we show that weakly stationary distributions exist for Markov processes with transition kernel satisfying Section \ref{assumptiondsf}. We start by giving an explicit construction of a weakly stationary distribution from the standard stationary distribution.

\begin{defn}\label{defwksta}
Let $\pi$ be the stationary distribution for $\kernel$. Let $\pi'$ be an internal probability measure on $(S,\mathcal{I}(S))$ satisfying
\begin{itemize}
\item For all $s\in S$, let $\pi'(\{s\})=\frac{\NSE{\pi}(B(s))}{\NSE{\pi}(\bigcup_{t\in S}B(t))}$.
\item For every internal set $A\subset S$, let $\pi'(A)=\sum_{s\in A}\pi'(\{s\})$.
\end{itemize}
\end{defn}

It is straightforward to verify from \ref{defwksta} that
\[ \label{EqPiPrimeSt}
\pi'(A)\approx \NSE{\pi}(\bigcup_{s\in A}B(s))
\]
 for every $A\in \mathcal{I}(S)$.
The following theorem relates $\pi'$ and $\pi$.

\begin{thm}[{\citep[][Lemma ~8.15]{Markovpaper}}]\label{starepresent}
$\pi'$ is an internal probability measure on $(S,\mathcal{I}(S))$. Moreover, for every $A\in \BorelSets X$, we have $\pi'(\ST^{-1}(A)\cap S)=\pi(A)$.
\end{thm}

We now show that $\pi'$ is a weakly stationary distribution for $\kerp$.

\begin{thm}\label{wkstationary}
Suppose $\{g(x,1,\cdot)\}_{x\in X}$ satisfies Section \ref{assumptiondsf}.
Let $\pi$ be the stationary distribution of $\kernel$.
Then $\pi'$ satisfying Definition \ref{defwksta} is a weakly stationary distribution for $\kerp$.
\end{thm}
\begin{proof}
Pick an internal set $A\subset S$ and some $t\in \Nats$.
By the transfer principle, we have $\pi'(A)\approx \NSE{\pi}(\bigcup_{a\in A}B(a))=\int \NSE{g}(x,t, \bigcup_{a\in A}B(a))\NSE{\pi}(\dee x)$.
Pick $\epsilon>0$, there is a compact set $K\subset X$ such that $\sum_{s\in S}\pi'(\{s\})G_{s}^{(t)}(A)-\sum_{s\in \NSE{K}\cap S}\pi'(\{s\})G_{s}^{(t)}(A)<\epsilon$ and
$\int \NSE{g}(x,t, \bigcup_{a\in A}B(a))\NSE{\pi}(\dee x)-\int_{S_{K}} \NSE{g}(x,t, \bigcup_{a\in A}B(a))\NSE{\pi}(\dee x)<\epsilon$, where $S_K=\bigcup_{s\in \NSE{K}\cap S}B(s)$.
As our choice of $\epsilon$ is arbitrary, to show $\pi'(A)\approx \sum_{s\in S}\pi'(\{s\})G_{s}^{(t)}(A)$, it is sufficient to show that $\sum_{s\in \NSE{K}\cap S}\pi'(\{s\})G_{s}^{(t)}(A)\approx \int_{S_{K}} \NSE{g}(x,t, \bigcup_{a\in A}B(a))\NSE{\pi}(\dee x)$.

Note that we have
\[
\int_{S_{K}} \NSE{g}(x,t, \bigcup_{a\in A}B(a))\NSE{\pi}(\dee x)&=\int_{\bigcup_{s\in \NSE{K}\cap S}B(s)}\NSE{g}(x,t, \bigcup_{a\in A}B(a))\NSE{\pi}(\dee x)\\
&=\sum_{s\in \NSE{K}\cap S}\int_{B(s)}\NSE{g}(x,t,\bigcup_{a\in A}B(a))\NSE{\pi}(\dee x).
\]
By Lemma \ref{dsfconsequence}, we have
\[ \label{ApprEqBsNse}
\int_{B(s)}\NSE{g}(x,t,\bigcup_{a\in A}B(a))\NSE{\pi}(\dee x)\approx \int_{B(s)}\NSE{g}(s,t,\bigcup_{a\in A}B(a))\NSE{\pi}(\dee x).
\]
As $\sum_{s\in \NSE{K}\cap S}\NSE{\pi}(B(s))<1$, by Theorem \ref{Gapprox}, we have
\[
&\sum_{s\in \NSE{K}\cap S}\int_{B(s)}\NSE{g}(x,t,\bigcup_{a\in A}B(a))\NSE{\pi}(\dee x)\\
&\stackrel{\text{Eq. } \eqref{ApprEqBsNse}}{\approx} \sum_{s\in \NSE{K}\cap S}\int_{B(s)}\NSE{g}(s,t,\bigcup_{a\in A}B(a))\NSE{\pi}(\dee x)\\
&=\sum_{s\in \NSE{K}\cap S}\NSE{g}(s,t,\bigcup_{a\in A}B(a))\NSE{\pi}(B(s))\\
&\stackrel{ Lemma \ref{dsfconsequence}}{\approx} \sum_{s\in \NSE{K}\cap S}G_{s}^{(t)}(A)\NSE{\pi}(B(s))\\
&\stackrel{\text{Eq. } \eqref{EqPiPrimeSt}}{\approx} \sum_{s\in \NSE{K}\cap S}G_{s}^{(t)}(A)\pi'(\{s\}).
\]
Thus, we can conclude that $\pi'(A)\approx \sum_{s\in S}\pi'(\{s\})G_{s}^{(t)}(A)$ for every $A\in \mathcal{I}(S)$ and every $t\in \Nats$.

Let $\cD=\{t\in T: (\forall A\in \mathcal{I}(S))(|\pi'(A)-\sum_{s\in S}\pi'(\{s\})G_{s}^{(t)}(A)|<\frac{1}{t})\}$. Then $\cD$ contains every $t\in \Nats$.
By overspill, there exists an infinite $t_0\in T$ such that $|\pi'(A)-\sum_{s\in S}\pi'(\{s\})G_{s}^{(t)}(A)|<\frac{1}{t}$ for all $A\in \mathcal{I}(S)$ and all $t\leq t_0$.
Thus, we have $\pi'(A)\approx \sum_{s\in S}\pi'(\{s\})G_{s}^{(t)}(A)$ for all $A\in \mathcal{I}(S)$ and all $t\leq t_0$.
\end{proof}

\subsection{Hyperfinite Representation of Reversible Markov Processes}\label{secreverse}
Recall that a Markov process is \emph{reversible} if it satisfies Equation \eqref{EqDefRev}, and in particular A Markov chain on a finite state space is reversible if and only if
\[
\pi(\{i\})g(i,1,\{j\})=\pi(\{j\})g(j,1,\{i\})
\]
for every $i,j$ in the state space $X$. If $\{g(x,1,\cdot)\}_{x\in X}$ is reversible, then its hyperfinite representation $\kerp$ is ``almost" reversible in the sense that


\[
\sum_{s\in S_1}G_{s}^{(t)}(S_2)\pi'(\{s\})\approx \sum_{s\in S_2}G_{s}^{(t)}(S_1)\pi'(\{s\})
\]
for every $S_1,S_2\in \mathcal{I}(S)$ and every $t\in \Nats$.
We now show that $\kerp$ is ``infinitesimally close" to a *reversible process.

\begin{thm}\label{closereverse}
Suppose $\{g(x,1,\cdot)\}_{x\in X}$ is reversible with stationary measure $\pi$ and satisfies Section \ref{assumptiondsf}.
Then there exists an internal transition kernel $\{H_{s}(\cdot)\}_{s\in S}$ such that it is *reversible with respect to $\pi'$ and 
\[
\max_{s\in S}\parallel G^{(t)}_{s}(\cdot)-H^{(t)}_{s}(\cdot) \parallel\approx 0.
\]
for every $t\in \Nats$.
\end{thm}
\begin{proof}
For every $x\in \NSE{X}$ and $A\in \NSE{\BorelSets X}$, define 
\[
F(x,A)=\delta(x,A)\NSE{g}(x,1,\NSE{X}\setminus \bigcup_{s\in S}B(s)) 
\]
for notational convenience. 
Thus, for every $i,j\in S$, by Definition \ref{dtrunchain}, we have $G_{ij}=\NSE{g}(i,1,B(j))+F(i,B(j))$.
Note that $\pi'(\{i\})=\frac{\NSE{\pi}(B(i))}{\NSE{\pi}(\bigcup_{s\in S}B(s))}$.

For every $i,j\in S$ with $\pi'(\{i\})\neq 0$, define
\[
H_{ij}=\frac{\int_{B(i)}\NSE{g}(x,1,B(j))+F(x,B(j))\frac{\NSE{\pi}(\dee x)}{\NSE{\pi}(\bigcup_{s\in S}B(s))}}{\pi'(\{i\})}.
\]
For $i$ with  $\pi'(\{i\})=0$, set $H_{ij}=G_{ij}=\NSE{g}(i,1,B(j))+F(i,B(j))$.
For every $i\in S$ and $A\in \mathcal{I}(S)$, define $H_{i}(A)=\sum_{j\in A}H_{ij}$.
It is straightforward to verify that $H_{i}(\cdot)$ defines an internal probability measure on $(S, \mathcal{I}(S))$ for every $i\in S$.
We now show that the hyperfinite Markov process with internal transition kernel $\{H_i(\cdot)\}_{i\in S}$ is *reversible.

\begin{claim}\label{starproperty}
The internal transition matrix $\{H_i(\cdot)\}_{i\in S}$ is *reversible with *stationary distribution $\pi'$.
\end{claim}
\begin{proof}
We start by showing that $\pi'$ is a *stationary distribution of $\{H_i(\cdot)\}_{i\in S}$.
Let $S_0=\{s\in S: \NSE{\pi}(B(s))>0\}$.
For $s\in S\setminus S_0$, note that $\int_{B(s)}\NSE{g}(x,1,B(j))\NSE{\pi}(\dee x)=0$ for every $j\in S$.
Thus, for every $j\in S$, we have
\[
&\sum_{i\in S}\pi'(\{i\})H_{ij}\\
&=\frac{1}{\NSE{\pi}(\bigcup_{s\in S}B(s))}\sum_{i\in S_{0}}\pi'(\{i\})\frac{\int_{B(i)}\NSE{g}(x,1,B(j))+F(x,B(j))\NSE{\pi}(\dee x)}{\pi'(\{i\})}\\
&=\frac{1}{\NSE{\pi}(\bigcup_{s\in S}B(s))}\sum_{i\in S}\int_{B(i)}\NSE{g}(x,1,B(j))+F(x,B(j))\NSE{\pi}(\dee x)\\
&=\frac{1}{\NSE{\pi}(\bigcup_{s\in S}B(s))}(\NSE{\pi}(B(j))-\int_{\NSE{X}\setminus \bigcup_{s\in S}B(s)}\NSE{g}(x,1,B(j))\NSE{\pi}(\dee x)+
\int_{B(j)}F(x,B(j))\NSE{\pi}(\dee x))\\
&=\frac{\NSE{\pi}(B(j))}{\NSE{\pi}(\bigcup_{s\in S}B(s))}=\pi'(\{j\}).
\]
Hence $\pi'$ is a *stationary distribution of $\{H_i(\cdot)\}_{i\in S}$.

We now show that the hyperfinite Markov process with internal transition kernel $\{H_i(\cdot)\}_{i\in S}$ is *reversible with respect its *stationary distribution $\pi'$.
For $t\in S\setminus S_0$, we have
\[
\pi'(\{t\})H_{tj}=0=\frac{1}{\NSE{\pi}(\bigcup_{s\in S}B(s))}\int_{B(t)}\NSE{g}(x,1,B(j))+F(x,B(j))\NSE{\pi}(\dee x)
\]
for every $j\in S$.
Thus, for every $i,j\in S$, we have
\[
&\pi'(\{i\})H_{ij}\\
&=\frac{1}{\NSE{\pi}(\bigcup_{s\in S}B(s))}\int_{B(i)}\NSE{g}(x,1,B(j))+F(x,B(j))\NSE{\pi}(\dee x)\\
&=\frac{1}{\NSE{\pi}(\bigcup_{s\in S}B(s))}(\int_{B(j)}\NSE{g}(x,1,B(i))\NSE{\pi}(\dee x)+\int_{B(j)}F(x,B(i))\NSE{\pi}(\dee x))\\
&=\frac{1}{\NSE{\pi}(\bigcup_{s\in S}B(s))}(\int_{B(j)}\NSE{g}(x,1,B(i))+F(x,B(i))\NSE{\pi}(\dee x)\\
&=\pi'(\{j\})H_{ji}.
\]
\end{proof}

We now prove the theorem by induction on $t\in \Nats$. Let $t=1$.
Pick $s\in S$ and $A\in \mathcal{I}(S)$.
If $\NSE{\pi}(B(s))=0$, then we have $G_{s}(A)=H_{s}(A)$.
Suppose $\NSE{\pi}(B(s))\neq 0$.
Pick $m\in \Nats$. We have
\[
&|G_{s}(A)-H_{s}(A)|\\
&\leq \frac{\int_{B(s)}|\NSE{g}(s,1,\bigcup_{a\in A}B(a))+F(s,\bigcup_{a\in A}B(a))-\NSE{g}(x,1,\bigcup_{a\in A}B(a))-F(x,\bigcup_{a\in A}B(a))|\NSE{\pi}(\dee x)}{\NSE{\pi}(\dee x)}\\
&\stackrel{Lemma \ref{dsfconsequence}}{\leq} \frac{\frac{1}{m}\NSE{\pi}(B(s))}{\NSE{\pi}(B(s))}=\frac{1}{m}.
\]
As our choices of $s$ and $A$ are arbitrary, we have $\max_{s\in S} \parallel G_{s}(\cdot)-H_{s}(\cdot) \parallel\approx 0$.

Suppose that the theorem holds for $t=n$. We now establish the result for $t=n+1$.
Pick $s\in S$ and $A\in \mathcal{I}(S)$.
By Lemma \ref{tvfunction}, we have
\[
&|G_{s}^{(n+1)}(A)-H_{s}^{(n+1)}(A)|\\
&=|\sum_{i\in S}G_{si}G_{i}^{(n)}(A)-\sum_{i\in S}H_{si}H_{i}^{(n)}(A)|\\
&\leq \parallel G_{s}(\cdot)-H_{s}(\cdot) \parallel\approx 0.
\]
As our choices of $s$ and $A$ are arbitrary, we have $\max_{s\in S} \parallel G_{s}^{(t+1)}(\cdot)-H_{s}^{(t+1)}(\cdot) \parallel\approx 0$, completing the proof.
\end{proof}

Throughout the paper, we shall denote the hyperfinite Markov process on $S$ with the internal transition matrix $\{H_{ij}\}_{i,j\in S}$ by $\{Z_{t}\}_{t\in T}$.
As the total variation distance between $\kerp$ and $\{H_i\}_{i \in S}$ is infinitesimal, it is not surprising that $\{H_i\}_{i \in S}$ can be used as a hyperfinite representation of $\kerp$. The following two theorems follow easily from Theorem \ref{Gapprox}, \ref{disrepresent} and \ref{closereverse} hence proofs are omitted. 

\begin{thm}\label{hGapprox}
Suppose $\{g(x,1,\cdot)\}_{x\in X}$ satisfies Section \ref{assumptiondsf}.
Then for any $t\in \Nats$, any $s\in \NS{S}$ and any $A\in {^{*}\BorelSets X}$, we have ${^{*}g}(s,t,\bigcup_{a\in A\cap S}B(a))\approx H_{s}^{(t)}(A\cap S)$.
\end{thm}

\begin{thm}\label{hdisrepresent}
Suppose $\{g(x,1,\cdot)\}_{x\in X}$ satisfies Section \ref{assumptiondsf}.
Then for any $s\in \NS{S}$, any $t\in \Nats$ and any $E\in \BorelSets X$, $g(\ST(s),t,E)=\overline{H}_{s}^{(t)}(\ST^{-1}(E)\cap S)$.
\end{thm}

Define the lazy transition kernel $\{I_{ij}\}_{i,j \in S}$ associated with $\{H_{ij}\}_{i,j \in S}$  to be a collection of internal transition probabilities satisfying the initial conditions $I^{(0)}_{ij}=H^{(0)}_{ij}$ and $I_{ij}=\frac{1}{2}H_{ij}+\frac{1}{2}\Delta(i,j)$, where $\Delta(i,j)=1$ if $i=j$ and $\Delta(i,j)=0$ if $i\neq j$. For every $i\in S$ and $A\in \mathcal{I}(S)$ we then have
\[
I_{i}(A) \equiv \sum_{j\in A}I_{ij}=\sum_{j\in A}(\frac{1}{2}H_{ij}+\frac{1}{2}\Delta(i,j))=\frac{1}{2}H_{i}(A)+\frac{1}{2}\Delta(i,A)
\]
where $\Delta(i,A)=1$ if $i\in A$ and $\Delta(i,A)=0$ if $i\not\in A$.

The following result shows that the total variation distance between the lazy chain of $\kerp$ and the lazy chain of $\{H_{i}\}_{i \in S}$ is infinitesimal.

\begin{lemma}\label{lazyclose}
Suppose $\{g(x,1,\cdot)\}_{x\in X}$ is reversible and satisfies Section \ref{assumptiondsf}.
Then we have
\[
\max_{s\in S}\parallel L^{(t)}_{s}(\cdot)-I^{(t)}_{s}(\cdot) \parallel\approx 0
\]
for every $t\in \Nats$.
\end{lemma}
\begin{proof}
We prove the result by induction on $t\in \Nats$.
Let $t=1$. By \ref{closereverse} and the construction of the lazy chain, we have
\[
\max_{s\in S} \parallel L_{s}(\cdot)-I_{s}(\cdot) \parallel \approx 0.
\]

Assume that the theorem holds for $t=n$. We now prove the case for $t=n+1$.
Pick $s\in S$ and $A\in \mathcal{I}(S)$. By Lemma \ref{tvfunction}, we have
\[
&|L_{s}^{(n+1)}(A)-I_{s}^{(n+1)}(A)|\\
&=|\sum_{i\in S}L_{si}L_{i}^{(n)}(A)-\sum_{i\in S}I_{si}I_{i}^{(n)}(A)|\\
&\leq \parallel L_{s}(\cdot)-I_{s}(\cdot) \parallel\approx 0.
\]
As our choices of $s$ and $A$ are arbitrary, we have $\max_{s\in S} \parallel L_{s}^{(t+1)}(\cdot)-I_{s}^{(t+1)}(\cdot) \parallel\approx 0$, completing the proof.
\end{proof}

It is not surprising that the lazy transition kernel $\{I_{s}(\cdot)\}_{s\in S}$ is a hyperfinite representation of the standard lazy transition kernel $\{g_{L}(x,1,\cdot)\}_{x\in X}$. The following two results follow directly from \ref{lazystar}, \ref{lazystandard} and \ref{lazyclose}. 

\begin{thm}\label{hlazystar}
Suppose $\{g(x,1,\cdot)\}_{x\in X}$ satisfies Section \ref{assumptiondsf}.
Then for any $t\in \Nats$, any $x\in \NS{\NSE{X}}$ and any $A\in {^{*}\BorelSets X}$,
we have ${^{*}g_{L}}(x,t,\bigcup_{a\in A\cap S}B(a))\approx I_{s_{x}}^{(t)}(A\cap S)$ where $s_{x}$ is the unique element in $S$ such that $x\in B(s_x)$.
\end{thm}


\begin{thm}\label{hlazystandard}
Suppose $\{g(x,1,\cdot)\}_{x\in X}$ satisfies Section \ref{assumptiondsf}.
Then for every $x\in X$, every $t\in \Nats$ and every $E\in \BorelSets X$, we have $g_{L}(x,t,E)=\Loeb{I}_{x}^{(t)}(\ST^{-1}(E)\cap S)$.
\end{thm}

\section{Mixing Times and Hitting Times with Their Nonstandard Counterparts}\label{secmha}
.
In this section, we develop nonstandard notions of mixing and hitting times for hyperfinite Markov processes and we show that the nonstandard notions and standard notions agree with each other. We assume that the underlying state space $X$ is compact for some theorems in this section. Recall that $X$ is compact if and only if $\NSE{X}=\NS{\NSE{X}}$. 

\subsection{Agreement of Mixing Time}\label{secmix}
Let $\{X_t\}_{t\in \Nats}$ be a discrete time Markov process on a general state space $X$ with transition probabilities denoted by $\{g(x,t,A)\}_{x\in X,t\in \Nats, A\in \BorelSets X}$ and stationary distribution $\pi$.
%
The following theorem shows that the mixing time of the lazy chain is no greater than the mixing time of the hyperfinite lazy chain.

\begin{thm}\label{mixlemma}
Suppose $\gkernel$ satisfies Section \ref{assumptiondsf}.
For every $\epsilon\in \PosReals$, we have
\[\label{mixsteq}
t_{L}(\epsilon) \leq \min\{t\geq 0: \sup_{i\in S}\ST(\parallel I_{i}^{(t)}(\cdot)-\pi'(\cdot) \parallel) \leq \epsilon\}.
\]
\end{thm}
\begin{proof}

By the definition of the Loeb measure, \ref{hlazystandard} and \ref{starepresent}, we have
\[
&\sup_{i\in S}\ST(\parallel I_{i}^{(t)}(\cdot)-\pi'(\cdot) \parallel)\\
&\geq \sup_{i\in S}\sup_{A\in \BorelSets X}|\Loeb{I}_{i}^{(t)}(\ST^{-1}(A)\cap S)-\Loeb{\pi'}(\ST^{-1}(A)\cap S)|\\
&\geq \sup_{x\in X}\sup_{A\in \BorelSets X}|g_{L}(x,t,A)-\pi(A)|\\
&=\sup_{x\in X}\parallel g_{L}(x,t,\cdot)-\pi(\cdot) \parallel.
\]
Thus, $t_{L}(\epsilon)\leq \min\{t\in \Nats: \sup_{i\in S}\ST(\parallel I_{i}^{(t)}(\cdot)- \pi'(\cdot)\parallel) \leq \epsilon\}$ for all $\epsilon>0$. 
\end{proof}

The following result is an immediate consequence of Lemma \ref{mixlemma}.

\begin{cor}\label{mixcor}
Suppose $\gkernel$ satisfies Section \ref{assumptiondsf}.
For every $\epsilon\in \PosReals$, we have
\[
t_{L}(\epsilon)\leq\NSE{\min}\{t\in T: \NSE{\sup}_{i\in S}\parallel I_{i}^{(t)}(\cdot)-\pi'(\cdot) \parallel \leq \epsilon\}.
\]
\end{cor}
\begin{proof}
Pick $\epsilon\in \PosReals$.
If $\NSE{\sup}_{i\in S}\parallel I_{i}^{(t)}(\cdot)-\pi'(\cdot) \parallel \leq \epsilon$
then
\[
\sup_{i\in S}\ST(\parallel I_{i}^{(t)}(\cdot)-\pi'(\cdot) \parallel) \leq \epsilon.
\]
The result then follows from Lemma \ref{mixlemma}.
\end{proof}

\subsection{Agreement of Hitting Time}\label{sechit}
Let $\{X_t\}_{t\in \Nats}$ be a discrete time Markov process with transition probabilities $\{g(x,1,A)\}_{x\in X, A\in \BorelSets X}$ and initial distribution $\nu$.
By Kolmogorov existence theorem, there exists a probability measure $\Prob$ on $(X^{\Nats},\BorelSets X^{\Nats})$ such that
\[\label{prodeq}
&\Prob(X_0\in A_0,X_1\in A_1,\dotsc, X_n\in A_n)\\
&=\int_{A_0}\nu(\dee x_0)\int_{A_1}g(x_0,1,\dee x_1)\dotsc\int_{A_{n-1}}g(x_{n-1},1,A_n)g(x_{n-2},1,\dee x_{n-1})
\]
for all $n\in \Nats$ and all $A_0,A_1,\dotsc, A_n\in \BorelSets X$. We write $\Prob_{x}(\cdot)$ for the probability of an event conditional on $X_0=x$.
Let $\{H_{ij}\}_{i,j\in S}$ and $\mu$ denote the internal transition matrix and the initial internal distribution of a *reversible hyperfinite Markov process $\{Z_t\}_{t\in T}$ defined in \ref{secreverse}, respectively. 
By \ref{HMexist}, we have:

\begin{thm}\label{hyexist}
There exists a hyperfinite probability space $(\Omega,\mathcal{I}(\Omega),\IProb)$ such that
\[
\IProb(Z_{0}=i_0,Z_{1}=i_1,\dotsc,Z_{t}=i_t)=\mu(\{i_0\})H_{i_{0}i_{1}}H_{i_{1}i_{2}}\dotsc H_{i_{t-1}i_{t}}
\]
for all $t\in T$ and $i_0,i_1,\dotsc,i_t\in S$.
\end{thm}
We write $\IProb_{s}(\cdot)$ for the internal probability of an internal event conditional on $Z_0=s$.

The first hitting time $\tau_{A}$ of a set $A\in \BorelSets X$ for $\{X_t\}_{t\in \Nats}$ is $\min\{t>0: X_t\in A\}$.
It is straightforward to see that $\Prob_{x}(\tau_{A}=1)=g(x,1,A)$.
For $t\geq 1$, we have $\Prob_{x}(\tau_{A}=t+1)=\int_{X\setminus A} \Prob_{y}(\tau_{A}=t) \dee g(x,1,\dee y)$.
Similarly, the first internal hitting time $\tau'_{A}$ of an internal set $A\subset S$ is defined to be $\min\{t\in T: Z_t\in A\}$.
It is easy to verify that $\IProb_{s}(\tau'_{A}=1)=H_{s}(A)$ for every $s\in S$ and $A\in \mathcal{I}(S)$.
For $t>1$, $s\in S$ and $A\in \mathcal{I}(S)$, we have $\IProb_{s}(\tau'_{A}=t)=\sum_{s_1,s_2,\dotsc,s_{t-1}\in S\setminus A}H_{ss_{1}}H_{s_{1}s_{2}}\dotsc H_{s_{t-1}}(A)$.
Thus, for $t\geq 1$, we have $\IProb_{s}(\tau'_{A}=t+1)=\int_{S\setminus A} \IProb_{y}(\tau'(A)=t)H_{s}(\dee y)$.

In order to apply nonstandard extensions and the transfer principle more easily, we define $\cP: X\times \BorelSets X\times \Nats\to [0,1]$ to be $\cP(x,B,t)=\Prob_{x}(\tau_{B}=t)$ and define $\cQ: S\times \mathcal{I}(S)\times T\to \NSE{[0,1]}$ to be $\cQ(s,A,t)=\IProb_{s}(\tau'_{A}=t)$.

\begin{thm}\label{stopprob}
Suppose $\gkernel$ satisfies Section \ref{assumptiondsf}.
Moreover, assume that the state space $X$ is compact.
For every $x\in \NSE{X}$, every $A\in \mathcal{I}(S)$ and every $t\in \Nats$,
we have $\NSE{\cP}(x,\bigcup_{a\in A}B(a),t)\approx \cQ(s_{x},A,t)$ where $s_x$ is the unique element in $S$ with $x\in B(s_x)$.
\end{thm}
\begin{proof}
For $t=1$, by Section \ref{assumptiondsf} and Theorem \ref{hGapprox}, we have
\[
\NSE{\cP}(x,\bigcup_{a\in A}B(a),1)=\NSE{g}(x,1,\bigcup_{a\in A}B(a))\approx H_{s_{x}}(A)=\cQ(s_{x},A,1)
\]
for every $x\in \NSE{X}$ and every $A\in \mathcal{I}(S)$.

Fix $n \in \mathbb{N}$ and suppose we have $\NSE{\cP}(x,\bigcup_{a\in A}B(a),t)\approx \cQ(s_{x},A,t)$ for every $x\in \NSE{X}$, every $A\in \mathcal{I}(S)$ and every $t\leq n$.
We now prove the case where $t=n+1$. By the induction hypothesis, we have
\[
&\NSE{\cP}(x,\bigcup_{a\in A}B(a),n+1)\\
&=\int_{\NSE{X}\setminus \bigcup_{a\in A}B(a)} \NSE{\cP}(y,\bigcup_{a\in A}B(a),n)\NSE{g}(x,1,\dee y)\\
&=\int_{\bigcup_{s\in S\setminus A}B(s)} \NSE{\cP}(y,\bigcup_{a\in A}B(a),n)\NSE{g}(x,1,\dee y)\\
&\approx\int_{\bigcup_{s\in S\setminus A}B(s)} \cQ(s_y,A,n)\NSE{g}(x,1,\dee y).
\]
By Lemma \ref{tvfunction}, we have
\[
\int_{\bigcup_{s\in S\setminus A}B(s)} \cQ(s_y,A,n)\NSE{g}(x,1,\dee y)\\
\approx \int_{\bigcup_{s\in S\setminus A}B(s)} \cQ(s_y,A,n)\NSE{g}(s_x,1,\dee y)\\
=\sum_{s\in S\setminus A}\cQ(s,A,n)\NSE{g}(s_x,1,B(s)).
\]
As $X$ is compact, by Definition \ref{dtrunchain}, we have $\NSE{g}(s_x,1,B(s))=G_{s_{x}s}$.
Thus, we have
\[
\sum_{s\in S\setminus A}\cQ(s,A,n)\NSE{g}(s_x,1,B(s))=\sum_{s\in S\setminus A}\cQ(s,A,n)G_{s_{x}s}.
\]
By Lemma \ref{tvfunction} and \ref{closereverse}, we have
\[
&|\sum_{s\in S\setminus A}\cQ(s,A,n)G_{s_{x}s}-\sum_{s\in S\setminus A}\cQ(s,A,n)H_{s_{x}s}|\\
&\leq \parallel G_{s_{x}}(\cdot)-H_{s_{x}}(\cdot) \parallel \approx 0.
\]
Hence, we have $\NSE{\cP}(x,\bigcup_{a\in A}B(a),n+1)\approx \cQ(s_x, A,n+1)$, completing the proof.

\end{proof}



The following result shows that the large hitting time of the standard Markov process defined in \ref{largehit} is bounded from below by the large hitting time of its hyperfinite representation.

\begin{thm}\label{hitlemma}
Let $\alpha\in \PosReals$.
Suppose $\gkernel$ satisfies Section \ref{assumptiondsf}.
Moreover, assume that the state space $X$ is compact.
Then
\[
\tau_{g}(\alpha)\geq \min\{t\in T: \NSE{\inf}\{\sum_{k=1}^{t}\cQ(s,A,k): s\in S, A\in \mathcal{I}(S)\ \text{such that}\ \pi'(A)\geq \alpha\}> 0.9\},
\]
provided that $\tau_{g}(\alpha)$ exists.
\end{thm}
\begin{proof}
Pick $\alpha\in \PosReals$ and suppose $\tau_{g}(\alpha)$ exists.
By the transfer principle, we have
\[
\tau_{g}(\alpha)=\NSE{\min}\{t\in T: \NSE{\inf}\{\sum_{k=1}^{t}\NSE{\cP}(x,A,k): x\in \NSE{X}, A\in \NSE{\BorelSets X} \ \text{such that}\ \NSE{\pi}(A)\geq \alpha\}>0.9\}.
\]
For every $A\in \mathcal{I}(S)$ with $\pi'(A)>\alpha$, by Definition \ref{defwksta}, we have $\NSE{\pi}(\bigcup_{a\in A}B(a))>\alpha$.
Thus, for every $n\in \Nats$, we have
\[
&\NSE{\inf}\{\sum_{k=1}^{n}\NSE{\cP}(x,A,k): x\in \NSE{X}, A\in \NSE{\BorelSets X} \ \text{such that}\ \NSE{\pi}(A)\geq \alpha\}\\
&\leq \NSE{\inf}\{\sum_{k=1}^{n}\NSE{\cP}(s,\bigcup_{a\in A}B(a),k): s\in S, A\in \mathcal{I}(S) \ \text{such that}\ \pi'(A)\geq \alpha\}\\
&\lessapprox \NSE{\inf}\{\sum_{k=1}^{n}\cQ(s,A,k): s\in S, A\in \mathcal{I}(S) \ \text{such that}\ \pi'(A)\geq \alpha\}.
\]
As $\tau_{g}(\alpha)$ exists, we have
\[
\inf\{\sum_{k=1}^{\tau_{g}(\alpha)}\cQ(s,A,k): s\in S, A\in \mathcal{I}(S) \ \text{such that}\ \pi'(A)\geq \alpha\}>0.9.
\]
Hence, we have the desired result.
\end{proof}

\subsection{Mixing Times and Hitting Times on Compact Sets} \label{SecMixHitCompact}

In this section, we use techniques developed in previous sections to prove Theorem \ref{mixhit} for reversible Markov processes with compact state spaces. The following lemma is well-known (for completeness, a proof can be found in \ref{wellknownbound} in the appendix):

\begin{lemma}\label{maxhitless}
Let $0 < \alpha<\frac{1}{2}$.
Let $\mathcal{D}$ denote the collection of discrete time transition kernels with a stationary distribution on a $\sigma$-compact metric state space. 
Then there exists a universal constant $d'_{\alpha}$ such that, for every $\gkernel \in \mathcal{D}$, we have
\[
d'(\alpha)t_{H}(\alpha) \leq t_{L}.
\]
\end{lemma}

We now show prove our main result Theorem \ref{mixhit} in the special case that the underlying state space is compact:

\begin{thm}\label{mixhitcompact}
Let $0 < \alpha<\frac{1}{2}$.
Then there exist universal constants $d_{\alpha},d'_{\alpha}$ such that, for every $\gkernel \in \mathcal{C}$, we have
\[
d'_{\alpha}t_{H}(\alpha)\leq t_{L}\leq d_{\alpha}t_{H}(\alpha).
\]
\end{thm}
\begin{proof}

Suppose $t_{H}(\alpha)$ is infinite. By \ref{maxhitless}, we know that $t_{L}$ is infinite. Thus, the result follows immediately in this case.

Suppose $t_{H}(\alpha)$ is finite.
Let $c_{\alpha}$ be the constant given in Theorem \ref{fmixhit}.
Let $\{I_{i}(\cdot)\}_{i\in S}$ be the internal transition probability matrix defined after \ref{hdisrepresent}.
By \ref{starproperty}, we know that $\{I_{i}(\cdot)\}_{i\in S}$ is a *reversible process with *stationary distribution $\pi'$.
Let
\[
T_{L}=\NSE{\min}\{t\in T: \NSE{\sup}_{i\in S}\parallel I_{i}^{(t)}(\cdot)-\pi'(\cdot) \parallel \leq \epsilon\}.
\]
Let
\[
T_{g}(\alpha)=\NSE{\min}\{t\in T: \NSE{\inf}\{\sum_{k=1}^{t}\cQ(s,A,k): s\in S, A\in \mathcal{I}(S)\ \text{such that}\ \pi'(A)\geq \alpha\}> 0.9\}
\]
where $\cQ(s,A,k)$ is defined in \ref{sechit}.

By the transfer of \ref{mhequal}, we know that $T_{L}\leq 2c_{\alpha}T_{g}(\alpha)$.
By \ref{mixcor}, we have $t_{L}\leq T_{L}$.
By \ref{hitlemma}, we have $\tau_{g}(\alpha)\geq T_{g}(\alpha)$.
Thus, we have $t_{L}\leq 2c_{\alpha}\tau_{g}(\alpha)$.
Let $d_{\alpha}=20c_{\alpha}$.
By \ref{maxlarge}, we have $t_{L}\leq d_{\alpha}t_{H}(\alpha)$.
By \ref{maxhitless}, we have the desired result.
\end{proof}


%

\section{Mixing Times and Hitting Times on $\sigma$-Compact Sets}\label{secsigcomp}

We fix notation as in Section \ref{SecMixHitCompact}, but relax the assumption that $(X,d)$ is a compact metric space to the assumption that $(X,d)$ is a $\sigma$-compact metric space.
As before, all $\sigma$-algebras should be taken to be the usual Borel $\sigma$-algebra.

We recall the definition of the \textit{trace} of a Markov chain:

\begin{defn} \label{DefTraceChain}
Let $g$ be the transition kernel of a Markov chain on state space $X$ with stationary measure $\pi$ and Borel $\sigma$-field $\BorelSets X$. Let $S \in \BorelSets X$ have measure $\pi(S) > 0$.

Fix $x \in X$ and let $\{X_{t}\}_{t \geq 0}$ be a Markov chain with transition kernel $g$ and starting point $X_{0} = x$. Then define the sequence  $\{\eta_{i}\}_{i \in \Nats}$ by setting
\[
\eta_{0} = \min \{t \geq 0 \, : \, X_{t} \in S \}
\]

and recursively setting

\[
\eta_{i+1} = \min \{t > \eta_{i} \, : \, X_{t} \in S \}.
\]

We then define the \textit{trace} of $g$ on $S$ to be the Markov chain with transition kernel

\[ \label{EqDefTrace}
g^{(S)}(x,t,A) = \mathbb{P}_{x}[X_{\eta_{t}} \in A].
\]
\end{defn}

\begin{rem}
Suppose that the original transition kernel $g$ has stationary distribution $\pi$. For $S\in \BorelSets X$ with $\pi(S)>0$, the normalization of $\pi$ to the set $S$ is the stationary distribution of the trace transition kernel $g^{(S)}$. Moreover, if $g$ is ergodic and reversible with respect to the stationary distribution $\pi$, then $g^{(S)}$ is reversible with respect to the normalization of $\pi$ to the set $S$.  

\end{rem}

\begin{rem} \label{RemCoupChainTrace}
Note that Definition \ref{DefTraceChain} naturally constructs a coupling of $\{X_{t}\}_{t \in \Nats} \sim g$ and $\{X_{t}^{(S)}\}_{t \in \Nats} \sim \{g^{(S)}(x,1,\cdot)\}_{x\in S}$ on the same probability space.
\end{rem}

\begin{rem} \label{RemLazyAlt}
We give an alternative definition of the ``lazy" kernel from \ref{deflazy} that is similar to the coupling in Remark \ref{RemCoupChainTrace}. Let $\{\zeta_{i} \}_{i \in \Nats}$ be a sequence of i.i.d. geometric random variables with mean 2, and define $L(t) = \max \{i \, : \, \sum_{j=1}^{i} \zeta_{i} \leq t\}.$

Observe that the chain $\{X_{t}^{(L)}\}_{t \in \Nats}$ given by
\[\label{EqLazyRep}
X_{t}^{(L)} = X_{L(t)}
\] 
satisfies $X_{0}^{(L)} = x$ and $\{X_{t}^{(L)}\}_{t \in \Nats} \sim g_{L}$. 
\end{rem}

A simple coupling argument, expanded in \ref{SubsecTraceProp}, gives:

\begin{lemma}\label{tracedsf}
Let $g$ be a transition kernel with stationary measure $\pi$ that satisfies Section \ref{assumptiondsf}, and let $S \in \BorelSets X$ be a set with measure $\pi(S) > 0$. Then the trace $g^{(S)}$ of $g$ on $S$ also satisfies Section \ref{assumptiondsf}. 
\end{lemma}

For the rest of the section, let $\mathcal{K}(X)$ denote the collection of all compact subsets of $X$ that are also in $\BorelSets X$. The next theorem shows that the standardized mixing time of the original Markov chain is bounded by the supremum over standardized mixing times of associated trace chains.\footnote{We freely use here the fact that the operation taking a kernel to its associated ``lazy" kernel and the operation taking a kernel to its associated ``trace" kernel commute. We include a proof of this fact in \ref{LemmaTraceLazyCommute} of the appendix for completeness.}

\begin{lemma} \label{LemmaIneqMixComp} 
Let $g$ be the transition kernel of a Markov chain on state space $X$ with stationary measure $\pi$.
For $S \in \BorelSets X$ with $\pi(S) > 0$, denote by $\overline{t}_{m}^{(S)}$ the  standardized mixing time with respect to $g^{(S)}$.
Then
\[
\overline{t}_{m} \leq \sup_{S\in \mathcal{K}(X)}\overline{t}_{m}^{(S)}.
\]
\end{lemma}

\begin{proof}

By the definition of $\overline{t}_{m}$, for all $\epsilon > 0$ there exist some particular points $x,y\in X$ and a set $A\in \BorelSets X$ such that

\[ \label{IneqLimCompBdOne}
| g(x,t,A) - g(y,t,A) | > 0.25 + \epsilon
\]

for $t =\overline{t}_{m}-1$. Next, note that $\{g(x,n,\cdot), g(y,n,\cdot) \}_{n=0}^{\overline{t}_{m}}$ is a finite collection of measures, and in particular it is tight. Therefore, there exists a compact set $S$ such that $\max_{0 \leq n \leq \overline{t}_{m}}\max\{g(x,n,S), g(y,n,S)\}\geq 1 - \frac{\epsilon}{100 \overline{t}_{m}}$ and $x,y\in S$. Combining this with Inequality \eqref{IneqLimCompBdOne}, the transition probabilities $g^{(S)}$  satisfy

\begin{align}
&|g^{(S)}(x,t,A \cap S) - g^{(S)}(y,t,  A \cap S) |\\
&\geq |g(x,t,A)-g(y,t,A)|-|g(x,t,A)-g^{(S)}(x,t,A)|-|g(y,t,A)-g^{(S)}(y,t,A)|\\
&\geq | g(x,t,A) - g(y,t,A) |-\sum_{n=0}^{t}g(x,n,X\setminus S)-\sum_{n=0}^{t}g(y,n,X\setminus S)\\
&\geq | g(x,t,A) - g(y,t,A) |-2(t+1)\max_{0\leq n\leq t}\max\{g(x,n,X\setminus S),g(y,n,X\setminus S)\}\\
&> 0.25 + \frac{98}{100} \epsilon.
\end{align}

Thus, the mixing time of $g^{(S)}$ is also at least $\overline{t}_{m}$, so we conclude

\[
\overline{t}_{m} \leq \sup_{S\in \mathcal{K}(X)} \overline{t}_{m}^{(S)}.
\]
\end{proof}

By the coupling in \ref{RemCoupChainTrace}, we have:

\begin{lemma}\label{mhitboundlemma}
Let $0 < \alpha<\frac{1}{2}$.
Let $g$ be the transition kernel of a Markov chain on state space $X$ with stationary measure $\pi$.
For $S \in \BorelSets X$ with $\pi(S) > 0$, denote by $\tau_{g}^{(S)}(\alpha)$ the large hitting time with respect to $g^{(S)}$.
Then
\[
\tau_{g}(\alpha)\geq \sup_{S\in \mathcal{K}(X)} \tau_{g}^{(S)}(\alpha).
\]

\end{lemma}


We can now prove Theorem \ref{mixhit}, the main result of this section:

\begin{thm}\label{mixhitpf}
Let $0 < \alpha<\frac{1}{2}$.
Then there exist universal constants $0<a_{\alpha},a'_{\alpha} < \infty$ such that, for every $\gkernel \in \mathcal{M}$, we have
\[
a'_{\alpha}t_{H}(\alpha)\leq t_{L}\leq a_{\alpha}t_{H}(\alpha).
\]
\end{thm}
\begin{proof}
By \ref{maxhitless}, there exists a universal constant $a'_{\alpha}>0$ such that, for every $\gkernel\in \mathcal{M}$, we have $a'_{\alpha}t_{H}(\alpha)\leq t_{L}$.
Recall that $\mathcal{C}$ is the collection of discrete time reversible transition kernels with compact state space satisfying Section \ref{assumptiondsf}.
By Theorem \ref{mixhitcompact}, there exists a universal constant $d_{\alpha}>0$ such that, for every $\gkernel\in \mathcal{M}$, the mixing time of the lazy chain is bounded by $d_{\alpha}$ times the maximal hitting time.
For every $\gkernel\in \mathcal{M}$, by \ref{mixequivalent}, we have $t_{L}\leq 2\overline{t}_{L}$.
By \ref{LemmaTraceLazyCommute}, \ref{LemmaIneqMixComp}, \ref{tracedsf} and Theorem \ref{mixhitcompact}, we have
\[
\overline{t}_{L}\leq \sup_{S\in \mathcal{K}(X)}\overline{t}_{L}^{(S)}\leq d_{\alpha}\sup_{S\in \mathcal{K}(X)}\tau_{g}^{(S)}(\alpha).
\]
By \ref{mhitboundlemma} and \ref{maxlarge}, we have
\[
\sup_{S\in \mathcal{K}(X)}\tau_{g}^{(S)}(\alpha)\leq \tau_{g}(\alpha)\leq 10t_{H}(\alpha).
\]
Let $a_{\alpha}=20d_{\alpha}$. We have $t_{L}\leq a_{\alpha}t_{H}(\alpha)$ for every $\gkernel\in \mathcal{M}$.
\end{proof}

\section{Statistical Applications and Extensions}\label{statapp}

In this section, we give results that allow us apply our main result, Theorem \ref{mixhit}, to obtain useful bounds for various Markov chains that don't satisfy its main assumptions. Our main motivation is the study of Markov chain Monte Carlo (MCMC) algorithms. MCMC is ubiquitous in statistical computation, and in this context small mixing times correspond to efficient algorithms (see \textit{e.g.} \citep{brooks2011handbook} for an overview of MCMC, \citep{gelman1995bayesian} for applications, and \citep{meyn2012markov} for analyses). Very few algorithms used for MCMC satisfy the strong Feller condition Section \ref{assumptiondsf}.

We begin by showing in Section \ref{SecMHMod} that our results apply without change to the Metropolis-Hastings algorithm, one of the most popular algorithms in computational statistics. In Section \ref{SecAsfIntro}, we introduce a relaxation of the strong Feller condition Section \ref{assumptiondsf} and then show that this relaxed property is satisfied by many other MCMC chains. Appendix Section \ref{AppOtherExt} contains further applications.

\subsection{Strong Feller Functions of Metropolis-Hastings Chains} \label{SecMHMod}

We begin with the following definition of a large class of Metropolis-Hastings chains:

\begin{defn} [Metropolis-Hastings Chain] \label{EqDefMH}
Fix a distribution $\pi$ with continuous density $\rho$ supported on $\mathbb{R}^{d}$. Also fix a reversible kernel $\{q(x,1,\cdot)\}_{x\in \Reals^{d}}$ on $\mathbb{R}^{d}$ with stationary measure $\nu$. For every $x\in \Reals^{d}$, assume that $q(x,1,\cdot)$ has continuous density $q_{x}$ and $\nu$ has continuous density $\phi$. Define the \textit{acceptance function} by the formula
\[
\beta(x,y) = \min(1, \frac{\rho(y) q_{y}(x)}{\rho(x) q_{x}(y)}).
\]
Finally, define $g$ to be the transition kernel given by the formula
\[\label{gtranform}
g(x,1,A) = \int_{y \in A} q_{x}(y) \beta(x,y) dy + \delta(x,A) \int_{\mathbb{R}^{d}} q_{x}(y) (1 - \beta(x,y))dy.
\]

For a transition kernel of this form, define the constant
\[
\gamma = \inf_{x} \int_{\mathbb{R}^{d}} q_{x}(y) \beta(x,y) dy.
\]

\end{defn}

\begin{rem}
It is well-known that, under these conditions, $g$ will be reversible with stationary measure $\pi$ (see \textit{e.g.} \citep{chib1995understanding}).
\end{rem}

Let $g$ be a Metropolis-Hastings kernel of form given in \ref{EqDefMH}, and let $\{X_{t}\}_{t \in \Nats} \sim g$. Then define inductively $\eta_{0} = 0$ and 
\[
\eta_{i+1} = \min \{t > \eta_{i} \, : \, X_{t} \neq X_{\eta_{i}} \}.
\]
Define the \textit{skeleton} of $\{X_{t}\}_{t \in \Nats}$ by 
\[ \label{EqSkelDef}
Y_{t} = X_{\eta_{t}}.
\]
The process $\{(Y_{t}, \eta_{t})\}_{t \in \Nats}$ is a Markov chain. We denote by $g'$ its transition kernel, and $\pi'$ its stationary measure on $X \times \Nats$. We remark that it is easy to reconstruct $\{X_{t}\}_{t \in \Nats}$ from  $\{(Y_{t},\eta_{t})\}_{t \in \Nats}$. For this section only, denote by $t_{m}', t_{L}'$ and $t_{H}'(\alpha)$ the mixing time, lazy mixing time and maximum hitting time of $g'$.

We then have:

\begin{thm}\label{IneqMHAlt}
Let $\mathcal{B}$ be the collection of transition kernels of the form given in \ref{gtranform} with finite mixing time, and for which $q_{x}(y)$ is jointly continuous in $x,y$.
Then for all $0 < \alpha < \frac{1}{2}$, there exists a universal constant $0 < c_{\alpha} < \infty$ so that 
\[
t_{L}' \leq c_{\alpha} (1 - \delta)^{-1} t_{H}(\delta \alpha)
\]
for every $g\in \mathcal{B}$ and every $\delta >0$. 
\end{thm}

\begin{proof}

Since $g$ is of the form \ref{gtranform}, it is straightforward to see that $g'$ satisfies Section \ref{assumptiondsf}. Thus, one can apply Theorem \ref{mixhit} to show that, for any $0 < \hat{\alpha} < \frac{1}{2}$,
\[ \label{EqAppMainSkel}
t_{L}' \sim t_{H}'(\hat{\alpha}),
\] 
where (as in Theorem \ref{mixhit}) the implied constant depends on $\hat{\alpha}$.

Next, we must relate $t_{H}'$ to $t_{H}$. For $x \in X$, let $\lambda(x) = g(x,1,\{x\}^{c})$. For $\lambda \in (0,\infty)$, denote by $L_{\lambda}$ the law of the geometric random variable with success probability $\lambda$ and let $\Prob_{\lambda}$ denote its associated probability mass function. For $A \subset X \times \Nats$, we observe
\[
\pi'(A) = \sum_{n \in \Nats} \int_{X} L_{\lambda(x)}(n) \mathbf{1}_{(x,n) \in A} \pi(dx).
\] 

Fix a measurable set  $A' \subset X \times \Nats$ with stationary measure $\pi'(A') \geq \alpha$. Define the associated ``core" set $A \subset X$ by 
\[ \label{EqCoreDef}
A = \{x \in X \, : \, \Prob_{\lambda(x)}(\{n \, : \, (x,n) \in A' \}) \geq (1- \delta) \alpha \}.
\]

Since $\pi'(A') \geq \alpha$, we must have $\pi(A) \geq \delta \alpha$. For chains $\{X_{t}\}_{t \in \Nats}$ and $\{(Y_{t},\eta_{t})\}_{t \in \Nats}$ coupled as in Equation \eqref{EqSkelDef}, define the hitting times 
\[
\tau_{A} = \min \{t \, : \, X_{t} \in A \}, \, \tau_{A}' = \min \{t \, : \, Y_{t} \in A \}, \, \tau_{A'}' = \min \{t \, : \, (Y_{t},\eta_{t}) \in A' \}. \\
\]

By the definition of the ``core" set in Equation \eqref{EqCoreDef},
\[
\E_{x}[\tau_{A'}'] \lesssim (1-\delta)^{-1} \, \E_{x}[\tau_{A}']
\]
for all starting points $x \in X$. Under our coupling of $\{X_{t}\}_{t \in \Nats}$ and $\{(Y_{t},\eta_{t})\}_{t \in \Nats}$,  
\[
\E_{x}[\tau_{A}'] \leq \E_{x}[\tau_{A}]
\]
for all starting points $x \in X$. Combining these two inequalities, we have 
\[
\E_{x}[\tau_{A'}'] \lesssim (1-\delta)^{-1} \E_{x}[\tau_{A}]
\]
for all $A' \in \BorelSets X$ with $\pi'(A') \geq \alpha$ and all $x \in X$. furthermore, $\pi(A) \geq \delta \alpha$, so
\[
t_{H}'(\alpha) \lesssim (1 - \delta)^{-1} t_{H}(\delta \alpha).
\]
Combining this with Equation \eqref{EqAppMainSkel}, completes the proof.

\end{proof}

\subsection{Almost-Strong Feller Chains} \label{SecAsfIntro}

We don't know a general way to extend the trick in Section \ref{SecMHMod}. Fortunately for us, in the context of MCMC, the user does not usually care about the mixing time of a \textit{specific} Markov chain - it is enough to estimate the mixing time of \textit{some} Markov chain that is both fast and easy to implement. We give the mathematical results first, then explain their relevance to MCMC in Section \ref{SubsubsecUsingMain}.

\subsubsection{Generic Bounds}\label{SecASFGen}
Let $\{g(x,1,\cdot)\}_{x\in X}$ be the transition kernel of a Markov process.
For every $k\in \Nats$, denote by $g^{(k)}$ the transition kernel
\[
g^{(k)}(x,t,A) = g(x,kt, A)
\]
for every $x\in X$, $t\in \Nats$ and $A\in \BorelSets X$.
We call $\{g^{(k)}(x,1,\cdot)\}_{x\in X}$ the \emph{$k$-skeleton} of  $\{g(x,1,\cdot)\}_{x\in X}$. We will use the superscript $(k)$ to extend our notation for the kernel $g$ to the kernel $g^{(k)}$. For example, for every $\epsilon>0$, we use $\overline{t}_{m}^{(k)}(\epsilon)$ to denote the standardized mixing time of $g^{(k)}$. We observe some simple relationships between $g$ and $g^{(k)}$, with details in Appendix \ref{SecMixHitSkel} for completeness:

\begin{lemma}\label{LemmaASFClaim1}
For all $\epsilon > 0$ and all $k \in \Nats$,
\[
\overline{t}_{m}^{(k)}(\epsilon) = \lceil \frac{\overline{t}_{m}(\epsilon)}{k} \rceil.
\]
\end{lemma}

\begin{lemma}\label{LemmaASFClaim2}
For all $\alpha > 0$ there exists a constant $0 < C_{\alpha} < \infty$ so that for all $k \in \Nats$,
\[
t_{H}^{(k)}(\alpha) \leq C_{\alpha} \lceil \frac{1}{k} \overline{t}_{m} \rceil.
\]
\end{lemma}

Next, we give a definition that relaxes the strong Feller condition in a quantitatively-useful way. We first make a small remark on three operations on kernels that we've defined: the trace of a kernel on a set, the $k$-skeleton of a kernel, and the ``lazy" version of a kernel. As shown in \ref{LemmaTraceLazyCommute}, the ``trace" and ``lazy" transformations commute - the trace of the lazy chain is equal to the lazy version of the trace chain. However, the $k$-skeleton and ``lazy" transformations \textit{do not} generally commute. As such, we occasionally use parentheses in the following notation to emphasize the order in which these transformations occur, with subscripts taking precedence. For example, $g_{L}^{(k)}$ is the $k$-skeleton of the chain $g_{L}$, while $(g^{(k)})_{L}$ is the lazy version of $g^{(k)}$. This last chain is important, and so we introduce the shorthand
\[ \label{EqGDef}
G \equiv (g_{L}^{(k)})_{L}.
\]
We also define $T_{m}$, $T_{L}$, and $T_{H}$ to be the mixing time, lazy mixing time and maximum hitting time of $G$. 

\begin{defn}[$(k,C)$-almost Strong Feller] \label{DefASF}
For $k,C \in \Nats$, we say that a kernel $\{g(x,1,\cdot)\}_{x\in X}$ is \textit{$(k,C)$-almost strong Feller} if there exist kernels $\{G_{1}(x,1,\cdot), G_{2}(x,1,\cdot)\}_{x\in X}$  so that the following are satisfied:

\begin{enumerate}
\item $G_{1}$ is reversible and satisfies Section \ref{assumptiondsf}, and
\item For some 
\[ \label{IneqASFGoodApprX}
0 \leq p \leq \frac{1}{\asfc},
\] 
we have
\[ \label{IneqASFGoodAppr}
g_{L}^{(k)} = (1 - p)G_{1} + pG_{2}.
\]
\end{enumerate}
\end{defn}

For the rest of the paper, we let $\mathcal{E}(k,C)$ be the collection of $(k,C)$-almost strong Feller transition kernels on a $\sigma$-compact metric state space $X$. 

\begin{rem}
Any strong Feller chain is $(1,C)$-almost strong Feller for all $C \geq 0$. Our condition is inspired by the famous \textit{asymptotically strong Feller} condition of \cite{hairer2006ergodicity}.
\end{rem}

To lessen notation in the rest of this section, we use ``$x \lesssim y$" as shorthand for the longer phrase ``there exists a universal constant $D$  such that $x \leq D y$," and $x \sim y$ for ``$x \lesssim y$ and $y \lesssim x$." We use the ``prime" superscript to denote quantities related to chains drawn from $G_{1}$. For example, we denote by e.g. $t_{m}'$ the mixing time of $G_{1}$ and $t_{L}'$ the mixing time of its associated lazy chain.

We then have the main result of this section, which shows that $T_{m}$ is bounded from above by $t_{H}^{(k)}(\alpha)$ under condition \ref{DefASF}: 

\begin{thm} \label{ThmAsfMainBd}
There exists a universal constant $C_0$ such that,
for every $0 < \alpha < 0.5$,
there exists a universal constant $d_{\alpha}$ such that for all $C > C_{0}$, all $k \in \mathbb{N}$ and all $\{g(x,1,\cdot)\}_{x\in X}\in \mathcal{E}(k,C)$, we have
\[
d_{\alpha} T_{m} \leq \ell_{H}^{(k)}(\alpha),
\]
where $\ell_{H}^{(k)}$ denotes the maximum hitting time of the transition kernel $g_{L}^{(k)}$.
\end{thm}

\begin{proof}
Pick $0<\alpha<\frac{1}{2}$. We have $k,C \in \Nats$ and $g \in \mathcal{E}(k,C)$ as generic constants and transition kernels, and we let $G_{1}, G_{2},$ and $p$ be associated kernels and constant as in \ref{DefASF}. By \ref{LemmaASFClaim1},
\[
t_{L}^{(k)} \sim 1 + \frac{t_{L}}{k}.
\]

Applying \ref{LemmaElCompLazyMix}, the mixing time $T_{m}$ of $G=(g_{L}^{(k)})_{L}$ satisfies 
\[
T_{m} \lesssim t_{L}^{(k)} \sim 1 + \frac{t_{L}}{k}.
\]

Applying \ref{LemmaElPertMixing} and \ref{IneqASFGoodAppr}, there exists a constant $C_{0}>0$ so that for all $C > C_{0}$, all $k\in \Nats$ and $g \in \mathcal{E}(k,C)$, we have
\[\label{llequal}
T_{m} \sim t_{L}'.
\]
as well. We restrict ourselves to $C > C_{0}$ for the remainder of the proof. 

Since the transition kernel $\{G_1(x,1,\cdot)\}_{x\in X}$ satisfies Section \ref{assumptiondsf}, Theorem \ref{mixhit} gives
\[ \label{IneqStar1}
t_{L}' \sim t_{H}'(\alpha).
\]

Applying \ref{IneqStar1} with \ref{LemmaASFClaim2} and \ref{LemmaElPertHitting}, we further have

\[
t_{L}' \sim t_{H}'(\alpha) \lesssim \ell_{H}^{(k)}(\alpha).
\]

Combining this with Inequality \eqref{llequal} completes the proof.
\end{proof}

\subsubsection{Gibbs Samplers}\label{SecASFGibbs}

We will show that \ref{ThmAsfMainBd} can be used to obtain nontrivial mixing bounds related to the following class of Gibbs samplers:

\begin{defn} [Gibbs Sampler] \label{EqDefGibbs}
Fix a distribution $\pi$ with continuous density $\rho > 0$ on $\mathbb{R}^{d}$. For $x \in \mathbb{R}^{d}$,  $i \in \{1,2,\ldots,n\}$ and $z \in \mathbb{R}$, define
\[
\rho_{x,i}(z) = \frac{\rho(x[1],x[2],\ldots,x[i-1],z,x[i+1],\ldots,x[d])}{\int \rho(x[1],x[2],\ldots,x[i-1],y,x[i+1],\ldots,x[d]) dy},
\]
the $i$'th conditional distribution of $\rho$. Let $F_{x,i}$ be the CDF of $\rho_{x,i}$. We then define a Markov chain as follows.

Fix a starting point $X_{0} = x$. Let $i_{t} \stackrel{iid}{\sim} \mathrm{Unif}(\{1,2,\ldots,d\})$ and $U_{t} \stackrel{iid}{\sim} \mathrm{Unif}([0,1])$ be two i.i.d. sequences. We iteratively define $X_{t+1}$ by the equation
\[ \label{EqGibbsForwardMap}
X_{t+1} = (X_{t}[1],\ldots,X_{t}[i_{t}-1], F_{X_{t},i_{t}}^{-1}(U_{t}), X_{t}[i_{t}+1],\ldots,X_{t}[d]).
\]
We define the transition kernel $g$ by setting
\[ \label{EqGibbsMapToKern}
g(x,t,A) = \mathbb{P}_{x}[X_{t} \in A]
\]
where $\Prob$ is a product measure that generates this Markov process.  \eqref{EqGibbsForwardMap} is the usual ``forward mapping" representation of a ``random-scan" Gibbs sampler. Note that, since $\rho$ is continuous and nonzero everywhere, $F_{x,i}^{-1}(u)$ always contains exactly one element for $x \in \Reals^{d}$, $i \in \{1,2,\ldots,n\}$ and $u \in [0,1]$.

Under the same setting as \eqref{EqGibbsMapToKern}, we define the associated ``conditional" update kernels $\{ g^{(i)}\}_{1 \leq i \leq d}$ by their one-step transition probabilities:
\[
g^{(i)}(x,1,A) = \mathbb{P}_{x}[X_{1} \in A | i_{0} = i].
\]

\end{defn}

The MCMC literature has many variants of the Gibbs sampler, but we focus on this popular simple case. Before stating our main result, we recall that any sequence of transition kernels $g_{1},\ldots,g_{k}$ on the same space has a product kernel, which we denote $\prod_{j=1}^{k} g_{j}$. Informally, this product is obtained by ``proposing from these kernels in order"; see \textit{e.g.} Theorem 5.17 of \cite{FMP2} for a formal justification of the notation.  Our main result is:

\begin{lemma} \label{LemmaGibbsSamplersAsf}
Let $\mathcal{A}$ be the collection of transition kernels of the  form given in \ref{EqDefGibbs}, that also have finite mixing time. Then, for all $0 < C < \infty$, there exists a universal constant $K_{C}$ so that $\{g(x,1,\cdot)\}_{x\in X} \in \mathcal{A}$ is $(k,C)$-almost strong Feller for any $k \geq K_{C} d \log(t_{L})$. 
\end{lemma}

\begin{rem}
The condition $\rho(\theta) > 0$ for all $\theta \in \mathbb{R}^{d}$ is only used as a simple sufficient condition for the chain $G_{1}$ defined in the proof to satisfy Section \ref{assumptiondsf}. In many other situations, this can be checked directly.
\end{rem}

\begin{proof}
Throughout this proof, we fix $g \in \mathcal{A}$ and use notation from \ref{EqDefGibbs} freely. We begin by bounding the mixing time from below. For $x \in \mathbb{R}^{d}$, define the collection 
\[
H(x) = \{y \in \mathbb{R}^{d} \, : \, \exists \, n \in \{1,2,\ldots,d\} \, \text{ s.t. } \, y[n] = x[n] \}
\]
of vectors that share at least one entry with $x$. Since $\pi$ has a density $\rho$, we have for any $x \in \mathbb{R}^{d}$ that
\[
\pi(H(x)) = 0.
\]
Thus, for any $x \in \mathbb{R}^{d}$ and $t \in \mathbb{N}$, we have
\[ \label{IneqNonSingGibbs}
\| g(x,t,\cdot) - \pi(\cdot) \| \geq P[\cup_{s=0}^{t-1} \{i_{s}\} \neq \{1,2,\ldots,d\}].
\]
By the representation for the lazy chain in \ref{RemLazyAlt}, we also have 
\[
\| g_{L}(x,t,\cdot) - \pi(\cdot) \| \geq P[\cup_{s=0}^{t-1} \{i_{s}\} \neq \{1,2,\ldots,d\}].
\]

By the well-known ``coupon collector" bound (see the main theorem of  \cite{erdHos1961classical}), there exists some $d_{0} \in \Nats$ such that for all $d \geq d_{0}$, 
\[ \label{IneqCoupQuote}
P[\cup_{s=0}^{\frac{1}{2} d \log(d) - 1} \{i_{s}\} \neq \{1,2,\ldots,d\}] \geq 1 - e^{-d}.
\]
Putting together Inequalities \eqref{IneqNonSingGibbs} to \eqref{IneqCoupQuote}, this implies that there exists some universal constant $0 < c_{1} < \infty$ so that for all $d \in \Nats$ and all $g \in \mathcal{A}$ on $\mathbb{R}^{d}$,
\[ \label{IneqGibbsAsfcKeyBd}
t_{m}, \, t_{L} \geq c_{1} d \log(d).
\]

Denote by $k \in \Nats$ a constant that will be fixed later in the proof. Let $\{X_{t}\}_{t \in \Nats} \sim g$, let $L$ be the (random) function from Equation \eqref{ezlform}, and let $\{\zeta_{i} \}_{i \in \Nats}$ be the i.i.d. geometric(2) random variables used to construct $L$ in  Equation \eqref{ezlform}. Recall that $\{X_{L(t)}\}_{t \in \Nats} \sim g_{L}$. For this choice, define the event
\[
\mathcal{E} = \{\cup_{t=1}^{k} \{i_{L(t)}\} = \{1,2,\ldots,d\} \},
\]
and define the kernels $G_{1}, G_{2}$  by setting
\begin{align*}
G_{1}(x,1,A) &= \mathbb{P}_{x}[X_{L(k)} \in A | \mathcal{E}] \\
G_{2}(x,1,A) &= \mathbb{P}_{x}[X_{L(k)} \in A | \mathcal{E}^{c}].
\end{align*}

In the notation of Definition \ref{DefASF}, the constant $p$ associated with this choice of $k, G_{1}, G_{2}$ is
\[
p = P[\mathcal{E}].
\]
We observe that, for any fixed $j \in \{1,2,\ldots,d\}$ and $s \in \Nats$, 
\[
P[ j \in \cup_{t=0}^{s-1} \{i_{t}\}] = 1- (1 - \frac{1}{d})^{s}.
\] 
On the other hand, by Hoeffding's inequality, we have
\[
P[L(k) < \frac{k}{4}] \leq e^{- \frac{1}{4} k^{2}}
\]
for all $k \geq 4$.

Combining these two bounds,
\[ \label{IneqPUpGib1}
p \leq P[\cup_{t=0}^{\frac{k}{4}-1} \{i_{t}\} \neq \{1,2,\ldots,d\}] + P[L(k) < \frac{k}{4}] \leq d (1 - \frac{1}{d})^{\frac{k}{4}} + e^{- \frac{1}{4} k^{2}} \leq d e^{-\frac{k}{4d}} + e^{- \frac{1}{4} k^{2}}.
\]
Noting $k,d \geq 1$, we have:

\[ \label{IneqPUpGib}
p \leq 2d e^{-\frac{k}{4d}}.
\]

To satisfy Inequalities \eqref{IneqASFGoodAppr} and \eqref{IneqASFGoodApprX}, we just need our choice of $k$ to ensure that $p \leq \frac{1}{\asfc}$. Inspecting these inequalities, there exists a universal constant $K$ so that this inequality is satisfied as long as
\[
k > K d \log(\max(d, t_{L})).
\] 
On the other hand, by Inequality \eqref{IneqGibbsAsfcKeyBd}, there exists a universal constant $A$ so that
\[
t_{L} \geq A d \log(\max(d,t_{L})).
\]

Inspecting these final two bounds, we see that for all $k$ sufficiently large compared to $d \log(t_{L}) \lesssim t_{L}$, this choice of $k$, $p$ and $G_{1}$ satisfies \eqref{IneqASFGoodAppr} and \eqref{IneqASFGoodApprX}.

Next, we must check that $G_{1}$ is reversible. To see this, we begin by noting that $\mathcal{E}$ depends only on the sequence $\{i_{t}\}_{t \in \Nats}$ of ``index" variables in our forward-mapping representation. Next, we check that, even after conditioning on $\mathcal{E}$, these index variables have a certain exchangeability-like property. For $m \in \Nats$ and  any sequence $J \in \Reals^{n}$ with $n \geq m+1$, define the ``reversal" function
\[ \label{EqReverseOp}
w_{m}(J) = (J[m], J[m-1],\ldots,J[1],J[0]).
\]

Let $T=L(k)$.   
We observe that for any sequence $j_{0},j_{1},\ldots \in \{1,2,\ldots,d\}$, we have 
\[
\mathbb{P}[(i_{0}, i_{1},\ldots,i_{T}) = (j_{0},j_{1},\ldots,j_{T}) | \mathcal{E}] = \mathbb{P}[(i_{0},i_{1},\ldots,i_{T}) = w_{T}(j_{0},j_{1},\ldots,j_{T}) | \mathcal{E}]
\]
and so for any $m \in \Nats$ and $J \in \{1,2,\ldots,d\}^{m+1}$
\[
\mathbb{P}[\{T=m\} \cap \{(i_{0},i_{1},\ldots,i_{T}) = J \} | \mathcal{E}] = \mathbb{P}[\{T=m\} \cap \{(i_{0},i_{1},\ldots,i_{T}) = w_{m}(J) \} | \mathcal{E}].
\]

Then, for $\{X_{t}\}_{t \in \Nats} \sim g$ drawn according to the forward-mapping representation, and for any $x \in X$, we have 
\begin{align*}
G_{1}&(x,1,A) = \mathbb{P}_{x}[X_{L(k)} \in A | \mathcal{E}] \\
&= \sum_{m \geq 0} \sum_{J \in \{1,\ldots,d\}^{m+1}}  \mathbb{P}_{x}[X_{L(k)} \in A | \mathcal{E}; \, \{ T=m\} \cap \{(i_{0},i_{1},\ldots,i_{T}) = J\}] \\
& \qquad \times \mathbb{P}[\{ T=m\} \cap \{(i_{0},i_{1},\ldots,i_{T}) = J\}] \\
&= \frac{1}{2} \sum_{m \geq 0} \sum_{J \in \{1,\ldots,d\}^{m+1}}  (\mathbb{P}_{x}[X_{L(k)} \in A | \mathcal{E};  \, \{ T=m\} \cap \{(i_{0},i_{1},\ldots,i_{T}) = J\}] \\
&+\mathbb{P}_{x}[X_{L(k)} \in A | \mathcal{E};  \, \{ T=m\} \cap \{(i_{0},i_{1},\ldots,i_{T}) = w_{m}(J)\}] ) \\
&\qquad \times \mathbb{P}[ \, \{ T=m\} \cap \{(i_{0},i_{1},\ldots,i_{T}) = J\}]  | \mathcal{E}] \\
&=  \sum_{m \geq 0} \sum_{J \in \{1,\ldots,d\}^{m+1}} \frac{1}{2} (\prod_{\ell=0}^{m} g^{(J[\ell])}
 + \prod_{\ell=0}^{m} g^{(w_{m}(J)[\ell])}) \mathbb{P}[ \, \{ T=m\} \cap \{(i_{0},i_{1}\ldots,i_{T}) = J\}] | \mathcal{E}]. 
\end{align*}

Recalling that $g^{(i)}$ is $\pi$-reversible for every $i$, we see that $\frac{1}{2} (\prod_{m=1}^{k} g^{(J[m])} + \prod_{m=1}^{k} g^{(w(J)[m])})$ is the additive reversibilization of the kernel $\prod_{m=1}^{k} g^{(J[m])}$. Hence, $\frac{1}{2} (\prod_{m=1}^{k} g^{(J[m])} + \prod_{m=1}^{k} g^{(w(J)[m])})$ is itself $\pi$-reversible (see \textit{e.g.} the introduction of \cite{choi2017metropolis} for a careful presentation of the additive reversibilization and a general argument as to why it is reversible). Thus, $G_{1}$ is $\pi$-reversible.

Finally, the fact that $G_{1}$ satisfies Section \ref{assumptiondsf} follows immediately from the fact that the function \eqref{EqGibbsForwardMap} that gives the forward-mapping representation of $g$ is continuous.
\end{proof}

We conclude:

\begin{thm}\label{ThmGibbsConc}
Let $\mathcal{A}$ be the collection of transition kernels of the  form given in \ref{EqDefGibbs}, that also have finite mixing time and satisfy $\rho(\theta) > 0$ for all $\theta \in \mathbb{R}^{d}$. Then there exists a universal constant $0<K<\infty$ such that,
for all $0 < \alpha < 0.5$, there exist constants $0 < c_{\alpha} < \infty$, $0<c_{\alpha}'<\infty$ so that for all $\{g(x,1,\cdot)\}_{x\in X}\in \mathcal{A}$,
\[
c_{\alpha} T_{m} \leq  \ell_{H}^{(k)}(\alpha) \leq c_{\alpha}' (1 + \frac{t_{L}}{k})
\]
uniformly in $k \geq Kd \log(t_{L}) $.

\end{thm}

\begin{proof}
The first inequality follows immediately from \ref{ThmAsfMainBd} and \ref{LemmaGibbsSamplersAsf}. The second follows from \ref{LemmaASFClaim2} and \ref{mixequivalent}.
\end{proof}

\subsubsection{Using \ref{ThmGibbsConc} and \ref{MhThmMainCorr} }\label{SubsubsecUsingMain}

We note that there are two obvious differences obstacles to using our main results, \ref{ThmGibbsConc} and \ref{MhThmMainCorr}, for practical problems:

\begin{enumerate}
\item Both refer to the transition kernel $G$ rather than the original kernel $g$ of interest.
\item Both require some choice of $k$, which in turn requires some a-priori bound on the mixing time of $g_{L}$.
\end{enumerate}

The first is not a practical problem, as it is straightforward to sample from $G$:

\begin{enumerate}
\item Sample a Markov chain $\{X_{t}\}_{t \in \Nats} \sim g$.
\item Transform the sequence $\{X_{0},X_{1},X_{2},\ldots\}$ by repeating each element a number of times given by a geometric(2) random variable, resulting in the sequence $\{X_{0},X_{1}', X_{2}',\ldots\}$.
\item Take the $k$-skeleton of this sequence, $\{Y_{0},Y_{1},Y_{2},\ldots\} \equiv \{X_{0}',X_{k}',X_{2k}',\ldots\}$.
\item As in step \textbf{(2)}, transform the sequence $\{Y_{0},Y_{1},Y_{2},\ldots\}$ by repeating each element a number of times given by a geometric(2) random variable, resulting in the sequence $\{Y_{0},Y_{1}', Y_{2}',\ldots\} \sim G$.
\end{enumerate}

The second point is slightly more subtle. In practice, it is often possible to find a \textit{weak} upper bound on a mixing time, even if \textit{practical} upper bounds are much harder. For a typical example, the paper \cite{collins2013accessibility} finds a very generic upper bound on the mixing time of a family of Markov chains that includes many Gibbs samplers. For many well-studied target distributions on $\mathbb{R}^{d}$ such as the uniform distribution on the simplex, box or ball, the upper bounds in \cite{collins2013accessibility} are roughly of the form $t_{m} \lesssim e^{c_{1} d}$. On the other hand, after many years of careful study, the true mixing times of many of these Markov chains were shown to be polynomial in $d$ (see \textit{e.g.} \cite{vempala2005geometric}). 

In this situation, it was not too difficult to show an \textit{exponential} bound on the mixing time, but it was quite hard to show a \textit{polynomial} bound. This is exactly the situation in which \ref{ThmGibbsConc} and \ref{MhThmMainCorr} are useful. If one can show that \textit{e.g.} $t_{L} \lesssim 2^{d}$, one can then choose $k \approx d \log(d) \lesssim t_{L}$ and use \ref{ThmGibbsConc} or \ref{MhThmMainCorr} to obtain much sharper estimates on the mixing time.

\bibliographystyle{imsart-nameyear}
\bibliography{biblio,drifttex-blx,equivalent-blx,Mixhit_Longpaper_Arxiv_Head-blx,mixing-blx,old-biblio}

\begin{thebibliography}{40}

\bibitem[\protect\citeauthoryear{Addario-Berry and
  Roberts}{}]{addario2017mixing}
\begin{barticle}[author]
\bauthor{\bsnm{Addario-Berry},~\bfnm{Louigi}\binits{L.}} \AND
  \bauthor{\bsnm{Roberts},~\bfnm{Matthew~I}\binits{M.~I.}}
\btitle{Mixing time bounds via bottleneck sequences}.
\bjournal{J. Stat. Phys.}
\bpages{1--27}.
\bdoi{10.1007/s10955-017-1917-5}
\end{barticle}
\endbibitem

\bibitem[\protect\citeauthoryear{Aldous}{1982}]{aldous1982some}
\begin{barticle}[author]
\bauthor{\bsnm{Aldous},~\bfnm{David~J}\binits{D.~J.}}
(\byear{1982}).
\btitle{Some inequalities for reversible {M}arkov chains}.
\bjournal{J. London Math. Soc.}
\bvolume{2}
\bpages{564--576}.
\end{barticle}
\endbibitem

\bibitem[\protect\citeauthoryear{Aldous and Fill}{2002}]{aldous2002reversible}
\begin{bmisc}[author]
\bauthor{\bsnm{Aldous},~\bfnm{David}\binits{D.}} \AND
  \bauthor{\bsnm{Fill},~\bfnm{James~Allen}\binits{J.~A.}}
(\byear{2002}).
\btitle{{R}eversible {M}arkov {C}hains and {R}andom {W}alks on {G}raphs}.
\bnote{Unfinished monograph, recompiled 2014, available at
  \url{http://www.stat.berkeley.edu/~aldous/RWG/book.html}}.
\end{bmisc}
\endbibitem

\bibitem[\protect\citeauthoryear{Aldous, Lov{\'a}sz and
  Winkler}{1997}]{aldous1997mixing}
\begin{barticle}[author]
\bauthor{\bsnm{Aldous},~\bfnm{David}\binits{D.}},
  \bauthor{\bsnm{Lov{\'a}sz},~\bfnm{L{\'a}szl{\'o}}\binits{L.}} \AND
  \bauthor{\bsnm{Winkler},~\bfnm{Peter}\binits{P.}}
(\byear{1997}).
\btitle{Mixing times for uniformly ergodic {M}arkov chains}.
\bjournal{Stoch. Process. Their Appl.}
\bvolume{71}
\bpages{165--185}.
\end{barticle}
\endbibitem

\bibitem[\protect\citeauthoryear{Anderson}{1976}]{andersonisrael}
\begin{barticle}[author]
\bauthor{\bsnm{Anderson},~\bfnm{Robert~M.}\binits{R.~M.}}
(\byear{1976}).
\btitle{A non-standard representation for {B}rownian motion and {I}t\^{o}
  integration}.
\bjournal{Israel J. Math.}
\bvolume{25}
\bpages{15--46}.
\bdoi{10.1007/BF02756559}
\bmrnumber{0464380}
\end{barticle}
\endbibitem

\bibitem[\protect\citeauthoryear{Anderson}{1982}]{anderson87}
\begin{barticle}[author]
\bauthor{\bsnm{Anderson},~\bfnm{Robert~M.}\binits{R.~M.}}
(\byear{1982}).
\btitle{Star-finite representations of measure spaces}.
\bjournal{Trans. Amer. Math. Soc.}
\bvolume{271}
\bpages{667--687}.
\bdoi{10.2307/1998904}
\bmrnumber{654856}
\end{barticle}
\endbibitem

\bibitem[\protect\citeauthoryear{Anderson, Duanmu and
  Smith}{2019}]{anderson2019mixavg}
\begin{bmisc}[author]
\bauthor{\bsnm{Anderson},~\bfnm{Robert~M.}\binits{R.~M.}},
  \bauthor{\bsnm{Duanmu},~\bfnm{Haosui}\binits{H.}} \AND
  \bauthor{\bsnm{Smith},~\bfnm{Aaron}\binits{A.}}
(\byear{2019}).
\btitle{Mixing Times and Average Mixing Times for General {M}arkov Processes}.
\end{bmisc}
\endbibitem

\bibitem[\protect\citeauthoryear{Arkeryd, Cutland and Henson}{1997}]{NSAA97}
\begin{bbook}[author]
\beditor{\bsnm{Arkeryd},~\bfnm{Leif~O.}\binits{L.~O.}},
  \beditor{\bsnm{Cutland},~\bfnm{Nigel~J.}\binits{N.~J.}} \AND
  \beditor{\bsnm{Henson},~\bfnm{C.~Ward}\binits{C.~W.}}, eds.
(\byear{1997}).
\btitle{Nonstandard analysis}.
\bseries{NATO Advanced Science Institutes Series C: Mathematical and Physical
  Sciences}
\bvolume{493}.
\bpublisher{Kluwer Academic Publishers Group, Dordrecht}
\bnote{Theory and applications}.
\bdoi{10.1007/978-94-011-5544-1}
\bmrnumber{1603227}
\end{bbook}
\endbibitem

\bibitem[\protect\citeauthoryear{Bardenet, Doucet and
  Holmes}{2017}]{bardenet2017markov}
\begin{barticle}[author]
\bauthor{\bsnm{Bardenet},~\bfnm{R{\'e}mi}\binits{R.}},
  \bauthor{\bsnm{Doucet},~\bfnm{Arnaud}\binits{A.}} \AND
  \bauthor{\bsnm{Holmes},~\bfnm{Chris}\binits{C.}}
(\byear{2017}).
\btitle{On {M}arkov chain {M}onte {C}arlo methods for tall data}.
\bjournal{J. Machine Learning Res.}
\bvolume{18}
\bpages{1515--1557}.
\end{barticle}
\endbibitem

\bibitem[\protect\citeauthoryear{Basu, Hermon and
  Peres}{2015}]{basu2015characterization}
\begin{binproceedings}[author]
\bauthor{\bsnm{Basu},~\bfnm{Riddhipratim}\binits{R.}},
  \bauthor{\bsnm{Hermon},~\bfnm{Jonathan}\binits{J.}} \AND
  \bauthor{\bsnm{Peres},~\bfnm{Yuval}\binits{Y.}}
(\byear{2015}).
\btitle{Characterization of cutoff for reversible {M}arkov chains}.
In \bbooktitle{Proceedings of the twenty-sixth annual ACM-SIAM symposium on
  Discrete algorithms}
\bpages{1774--1791}.
\bpublisher{Society for Industrial and Applied Mathematics}.
\end{binproceedings}
\endbibitem

\bibitem[\protect\citeauthoryear{Bovier}{2006}]{bovier2006metastability}
\begin{binproceedings}[author]
\bauthor{\bsnm{Bovier},~\bfnm{Anton}\binits{A.}}
(\byear{2006}).
\btitle{Metastability: a potential theoretic approach}.
\end{binproceedings}
\endbibitem

\bibitem[\protect\citeauthoryear{Brooks et~al.}{2011}]{brooks2011handbook}
\begin{bbook}[author]
\bauthor{\bsnm{Brooks},~\bfnm{Steve}\binits{S.}},
  \bauthor{\bsnm{Gelman},~\bfnm{Andrew}\binits{A.}},
  \bauthor{\bsnm{Jones},~\bfnm{Galin}\binits{G.}} \AND
  \bauthor{\bsnm{Meng},~\bfnm{Xiao-Li}\binits{X.-L.}}
(\byear{2011}).
\btitle{Handbook of {M}arkov chain {M}onte {C}arlo}.
\bpublisher{CRC press}.
\end{bbook}
\endbibitem

\bibitem[\protect\citeauthoryear{Chib and
  Greenberg}{1995}]{chib1995understanding}
\begin{barticle}[author]
\bauthor{\bsnm{Chib},~\bfnm{Siddhartha}\binits{S.}} \AND
  \bauthor{\bsnm{Greenberg},~\bfnm{Edward}\binits{E.}}
(\byear{1995}).
\btitle{Understanding the {M}etropolis-{H}astings algorithm}.
\bjournal{The American Statistician}
\bvolume{49}
\bpages{327--335}.
\end{barticle}
\endbibitem

\bibitem[\protect\citeauthoryear{Choi}{2017}]{choi2017metropolis}
\begin{bmisc}[author]
\bauthor{\bsnm{Choi},~\bfnm{Michael~CH}\binits{M.~C.}}
(\byear{2017}).
\btitle{Metropolis-{H}astings reversiblizations of non-reversible {M}arkov
  chains}.
\end{bmisc}
\endbibitem

\bibitem[\protect\citeauthoryear{Collins
  et~al.}{2013}]{collins2013accessibility}
\begin{bmisc}[author]
\bauthor{\bsnm{Collins},~\bfnm{Benoit}\binits{B.}},
  \bauthor{\bsnm{Kousha},~\bfnm{Termeh}\binits{T.}},
  \bauthor{\bsnm{Kulik},~\bfnm{Rafa{\l}}\binits{R.}},
  \bauthor{\bsnm{Szarek},~\bfnm{Tomasz}\binits{T.}} \AND \bauthor{\bsnm{{\.
  Z}yczkowski},~\bfnm{Karol}\binits{K.}}
(\byear{2013}).
\btitle{The accessibility of convex bodies and derandomization of the hit and
  run algorithm}.
\end{bmisc}
\endbibitem

\bibitem[\protect\citeauthoryear{Cutland et~al.}{1995}]{NDV}
\begin{bbook}[author]
\beditor{\bsnm{Cutland},~\bfnm{Nigel~J.}\binits{N.~J.}},
  \beditor{\bsnm{Neves},~\bfnm{V{\'{\i}}tor}\binits{V.}},
  \beditor{\bsnm{Oliveira},~\bfnm{Franco}\binits{F.}} \AND
  \beditor{\bsnm{Sousa-Pinto},~\bfnm{Jos{\'e}}\binits{J.}}, eds.
(\byear{1995}).
\btitle{Developments in nonstandard mathematics}.
\bseries{Pitman Research Notes in Mathematics Series}
\bvolume{336}.
\bpublisher{Longman, Harlow}
\bnote{Papers from the International Colloquium (CIMNS94) held in memory of
  Abraham Robinson at the University of Aveiro, Aveiro, July 18--22, 1994}.
\bmrnumber{1394201}
\end{bbook}
\endbibitem

\bibitem[\protect\citeauthoryear{den Hollander}{}]{den2012probability}
\begin{bbook}[author]
\bauthor{\bparticle{den} \bsnm{Hollander},~\bfnm{Frank}\binits{F.}}
\btitle{Probability theory: The coupling method}.
\end{bbook}
\endbibitem

\bibitem[\protect\citeauthoryear{Diaconis}{2009}]{diaconis2009markov}
\begin{barticle}[author]
\bauthor{\bsnm{Diaconis},~\bfnm{Persi}\binits{P.}}
(\byear{2009}).
\btitle{The {M}arkov chain {M}onte {C}arlo revolution}.
\bjournal{Bulletin American Math. Soc.}
\bvolume{46}
\bpages{179--205}.
\end{barticle}
\endbibitem

\bibitem[\protect\citeauthoryear{Duanmu, Rosenthal and
  Weiss}{2018}]{Markovpaper}
\begin{bunpublished}[author]
\bauthor{\bsnm{Duanmu},~\bfnm{Haosui}\binits{H.}},
  \bauthor{\bsnm{Rosenthal},~\bfnm{J.~S.}\binits{J.~S.}} \AND
  \bauthor{\bsnm{Weiss},~\bfnm{William}\binits{W.}}
(\byear{2018}).
\btitle{Ergodicity of Markov processes via non-standard analysis}.
\bnote{Memoirs of the American Mathematical Society, to appear}.
\end{bunpublished}
\endbibitem

\bibitem[\protect\citeauthoryear{Duanmu and Roy}{2016}]{nsbayes}
\begin{bmisc}[author]
\bauthor{\bsnm{Duanmu},~\bfnm{Haosui}\binits{H.}} \AND
  \bauthor{\bsnm{Roy},~\bfnm{Daniel~M.}\binits{D.~M.}}
(\byear{2016}).
\btitle{On extended admissible procedures and their nonstandard Bayes risk}.
\bnote{Submitted}.
\end{bmisc}
\endbibitem

\bibitem[\protect\citeauthoryear{Erd{\H{o}}s}{1961}]{erdHos1961classical}
\begin{barticle}[author]
\bauthor{\bsnm{Erd{\H{o}}s},~\bfnm{Paul}\binits{P.}}
(\byear{1961}).
\btitle{On a classical problem of probability theory}.
\end{barticle}
\endbibitem

\bibitem[\protect\citeauthoryear{Gelman et~al.}{1995}]{gelman1995bayesian}
\begin{bbook}[author]
\bauthor{\bsnm{Gelman},~\bfnm{Andrew}\binits{A.}},
  \bauthor{\bsnm{Carlin},~\bfnm{John~B}\binits{J.~B.}},
  \bauthor{\bsnm{Stern},~\bfnm{Hal~S}\binits{H.~S.}} \AND
  \bauthor{\bsnm{Rubin},~\bfnm{Donald~B}\binits{D.~B.}}
(\byear{1995}).
\btitle{Bayesian data analysis}.
\bpublisher{Chapman and Hall/CRC}.
\end{bbook}
\endbibitem

\bibitem[\protect\citeauthoryear{Hairer and
  Mattingly}{2006}]{hairer2006ergodicity}
\begin{barticle}[author]
\bauthor{\bsnm{Hairer},~\bfnm{Martin}\binits{M.}} \AND
  \bauthor{\bsnm{Mattingly},~\bfnm{Jonathan~C}\binits{J.~C.}}
(\byear{2006}).
\btitle{Ergodicity of the 2D Navier-Stokes equations with degenerate stochastic
  forcing}.
\bjournal{Ann. Math.}
\bpages{993--1032}.
\end{barticle}
\endbibitem

\bibitem[\protect\citeauthoryear{Hermon}{2018}]{hermon2018sensitivity}
\begin{binproceedings}[author]
\bauthor{\bsnm{Hermon},~\bfnm{Jonathan}\binits{J.}}
(\byear{2018}).
\btitle{On sensitivity of uniform mixing times}.
In \bbooktitle{Annales de l'Institut Henri Poincar{\'e}, Probabilit{\'e}s et
  Statistiques}
\bvolume{54}
\bpages{234--248}.
\bpublisher{Institut Henri Poincar{\'e}}.
\end{binproceedings}
\endbibitem

\bibitem[\protect\citeauthoryear{Herv{\'e} and
  Ledoux}{2014}]{herve2014approximating}
\begin{barticle}[author]
\bauthor{\bsnm{Herv{\'e}},~\bfnm{Lo{\"\i}c}\binits{L.}} \AND
  \bauthor{\bsnm{Ledoux},~\bfnm{James}\binits{J.}}
(\byear{2014}).
\btitle{Approximating Markov chains and V-geometric ergodicity via weak
  perturbation theory}.
\bjournal{Stoch. Process. Their Appl.}
\bvolume{124}
\bpages{613--638}.
\end{barticle}
\endbibitem

\bibitem[\protect\citeauthoryear{Kallenberg}{2002}]{FMP2}
\begin{bbook}[author]
\bauthor{\bsnm{Kallenberg},~\bfnm{Olav}\binits{O.}}
(\byear{2002}).
\btitle{Foundations of modern probability},
\bedition{2nd} ed.
\bpublisher{Springer}, \baddress{New York}.
\bmrnumber{MR1876169 (2002m:60002)}
\end{bbook}
\endbibitem

\bibitem[\protect\citeauthoryear{Keisler}{1984}]{Keisler87}
\begin{barticle}[author]
\bauthor{\bsnm{Keisler},~\bfnm{H.~Jerome}\binits{H.~J.}}
(\byear{1984}).
\btitle{An infinitesimal approach to stochastic analysis}.
\bjournal{Mem. Amer. Math. Soc.}
\bvolume{48}
\bpages{x+184}.
\bdoi{10.1090/memo/0297}
\bmrnumber{732752}
\end{barticle}
\endbibitem

\bibitem[\protect\citeauthoryear{Levin, Peres and Wilmer}{2009}]{markovmix}
\begin{bbook}[author]
\bauthor{\bsnm{Levin},~\bfnm{David~A.}\binits{D.~A.}},
  \bauthor{\bsnm{Peres},~\bfnm{Yuval}\binits{Y.}} \AND
  \bauthor{\bsnm{Wilmer},~\bfnm{Elizabeth~L.}\binits{E.~L.}}
(\byear{2009}).
\btitle{Markov chains and mixing times}.
\bpublisher{American Mathematical Society, Providence, RI}
\bnote{With a chapter by James G. Propp and David B. Wilson}.
\bmrnumber{2466937}
\end{bbook}
\endbibitem

\bibitem[\protect\citeauthoryear{Loeb}{1975}]{Loeb75}
\begin{barticle}[author]
\bauthor{\bsnm{Loeb},~\bfnm{Peter~A.}\binits{P.~A.}}
(\byear{1975}).
\btitle{Conversion from nonstandard to standard measure spaces and applications
  in probability theory}.
\bjournal{Trans. Amer. Math. Soc.}
\bvolume{211}
\bpages{113--122}.
\bmrnumber{0390154}
\end{barticle}
\endbibitem

\bibitem[\protect\citeauthoryear{Meyn and Tweedie}{2012}]{meyn2012markov}
\begin{bbook}[author]
\bauthor{\bsnm{Meyn},~\bfnm{Sean~P}\binits{S.~P.}} \AND
  \bauthor{\bsnm{Tweedie},~\bfnm{Richard~L}\binits{R.~L.}}
(\byear{2012}).
\btitle{Markov chains and stochastic stability}.
\bpublisher{Springer Science \& Business Media}.
\end{bbook}
\endbibitem

\bibitem[\protect\citeauthoryear{Mitrophanov}{2005}]{mitrophanov2005sensitivity}
\begin{barticle}[author]
\bauthor{\bsnm{Mitrophanov},~\bfnm{A~Yu}\binits{A.~Y.}}
(\byear{2005}).
\btitle{Sensitivity and convergence of uniformly ergodic {M}arkov chains}.
\bjournal{J. Appl. Prob.}
\bvolume{42}
\bpages{1003--1014}.
\end{barticle}
\endbibitem

\bibitem[\protect\citeauthoryear{Montenegro and
  Tetali}{2006}]{montenegro2006mathematical}
\begin{barticle}[author]
\bauthor{\bsnm{Montenegro},~\bfnm{Ravi}\binits{R.}} \AND
  \bauthor{\bsnm{Tetali},~\bfnm{Prasad}\binits{P.}}
(\byear{2006}).
\btitle{Mathematical aspects of mixing times in {M}arkov chains}.
\bjournal{Foundations and Trends{\textregistered} in Theoretical Computer
  Science}
\bvolume{1}
\bpages{237--354}.
\end{barticle}
\endbibitem

\bibitem[\protect\citeauthoryear{Negrea and Rosenthal}{2017}]{negrea2017error}
\begin{bmisc}[author]
\bauthor{\bsnm{Negrea},~\bfnm{Jeffrey}\binits{J.}} \AND
  \bauthor{\bsnm{Rosenthal},~\bfnm{Jeffrey~S}\binits{J.~S.}}
(\byear{2017}).
\btitle{Error Bounds for Approximations of Geometrically Ergodic {M}arkov
  Chains}.
\end{bmisc}
\endbibitem

\bibitem[\protect\citeauthoryear{Oliveira}{2012}]{oliveira2012mixing}
\begin{barticle}[author]
\bauthor{\bsnm{Oliveira},~\bfnm{Roberto}\binits{R.}}
(\byear{2012}).
\btitle{Mixing and hitting times for finite {M}arkov chains}.
\bjournal{Elec. J. Prob.}
\bvolume{17}.
\end{barticle}
\endbibitem

\bibitem[\protect\citeauthoryear{Peres and Sousi}{2015}]{finitemixhit}
\begin{barticle}[author]
\bauthor{\bsnm{Peres},~\bfnm{Yuval}\binits{Y.}} \AND
  \bauthor{\bsnm{Sousi},~\bfnm{Perla}\binits{P.}}
(\byear{2015}).
\btitle{Mixing times are hitting times of large sets}.
\bjournal{J. Theoret. Probab.}
\bvolume{28}
\bpages{488--519}.
\bdoi{10.1007/s10959-013-0497-9}
\bmrnumber{3370663}
\end{barticle}
\endbibitem

\bibitem[\protect\citeauthoryear{Pillai and Smith}{2014}]{pillai2014ergodicity}
\begin{bmisc}[author]
\bauthor{\bsnm{Pillai},~\bfnm{Natesh~S}\binits{N.~S.}} \AND
  \bauthor{\bsnm{Smith},~\bfnm{Aaron}\binits{A.}}
(\byear{2014}).
\btitle{Ergodicity of approximate {MCMC} chains with applications to large data
  sets}.
\end{bmisc}
\endbibitem

\bibitem[\protect\citeauthoryear{Robinson}{1966}]{AR65}
\begin{bbook}[author]
\bauthor{\bsnm{Robinson},~\bfnm{Abraham}\binits{A.}}
(\byear{1966}).
\btitle{Non-standard analysis}.
\bpublisher{North-Holland Publishing Co., Amsterdam}.
\bmrnumber{0205854}
\end{bbook}
\endbibitem

\bibitem[\protect\citeauthoryear{Rudolf and
  Schweizer}{2018}]{rudolf2018perturbation}
\begin{barticle}[author]
\bauthor{\bsnm{Rudolf},~\bfnm{Daniel}\binits{D.}} \AND
  \bauthor{\bsnm{Schweizer},~\bfnm{Nikolaus}\binits{N.}}
(\byear{2018}).
\btitle{Perturbation theory for {M}arkov chains via {W}asserstein distance}.
\bjournal{Bernoulli}
\bvolume{24}
\bpages{2610--2639}.
\end{barticle}
\endbibitem

\bibitem[\protect\citeauthoryear{Vempala}{2005}]{vempala2005geometric}
\begin{barticle}[author]
\bauthor{\bsnm{Vempala},~\bfnm{Santosh}\binits{S.}}
(\byear{2005}).
\btitle{Geometric random walks: a survey}.
\bjournal{Combinatorial and Computational Geometry}
\bpages{573--612}.
\end{barticle}
\endbibitem

\bibitem[\protect\citeauthoryear{Wolff and Loeb}{2000}]{NAW}
\begin{bbook}[author]
\beditor{\bsnm{Wolff},~\bfnm{Manfred}\binits{M.}} \AND
  \beditor{\bsnm{Loeb},~\bfnm{Peter~A.}\binits{P.~A.}}, eds.
(\byear{2000}).
\btitle{Nonstandard analysis for the working mathematician}.
\bseries{Mathematics and its Applications}
\bvolume{510}.
\bpublisher{Kluwer Academic Publishers, Dordrecht}.
\bdoi{10.1007/978-94-011-4168-0}
\bmrnumber{1790871}
\end{bbook}
\endbibitem

\end{thebibliography}

\appendix

\section{Some Elementary Bounds}\label{secebound}

In this section, we prove a few useful bounds on mixing times and hitting times. Results in this section are used throughout the paper.

\subsection{Elementary Equivalences} \label{SecElEq}

\begin{proof} [Proof of Lemma \ref{maxlarge}]

Pick $\alpha\in \PosReals$.
First, suppose $0 \leq \tau_{g}(\alpha) < \infty$ exists.
Pick $x_0\in X$ and $A_0\in \BorelSets X$ with $\pi(A_0)\geq \alpha$ such that $\Prob_{x_0}(\tau_{A_0}\leq \tau_{g}(\alpha)-1)\leq 0.9$. Then we have
\[
\expect_{x_0}(\tau_{A_0})\geq \tau_{g}(\alpha)\Prob_{x_0}(\tau_{A_0}>\tau_{g}(\alpha)-1)\geq 0.1\tau_{g}(\alpha).
\]
Thus, we immediately have $t_{H}(\alpha)\geq 0.1\tau_{g}(\alpha)$.

Conversely, by a regeneration argument, we have
\[
\inf\{\Prob_{x}(\tau_{A}\leq k\tau_{g}(\alpha)): x\in X, A\in \BorelSets X \ \text{such that}\ \pi(A)\geq \alpha\}>1-0.1^{k}
\]
for all $k\in \Nats$. Thus, for every $x\in X$ and every $A\in \BorelSets X$ with $\pi(A)\geq \alpha$, we have
\[
\expect_{x}(\tau_{A})&\leq \sum_{k=0}^{\infty}\tau_{g}(\alpha)\Prob_{x}(\tau_{A}>k \tau_{g}(\alpha))\\
&\leq \tau_{g}(\alpha)\sum_{k=0}^{\infty}0.1^{k}\\
&\leq 2\tau_{g}(\alpha).
\]
Thus, we have $t_{H}(\alpha)\leq 2\tau_{g}(\alpha)$.

On the other hand, the same calculations show that $\tau_{g}(\alpha)$ is infinite if and only if $t_{H}(\alpha)$ is infinite.
\end{proof}

\subsection{Mixing and Hitting Times of Perturbed Chains}

\begin{lemma} [Submultiplicative Bounds on Hitting Times] \label{LemmaElSub}
Let $\{g(x,1,\cdot)\}_{x\in X}$ be a transition kernel on a $\sigma$-compact metric state space $X$ with stationary distribution $\pi$.  
Fix a constant $\alpha > 0$ and set $A \in \BorelSets X$ with $\pi(A)>\alpha$. 
Then for all $k \in \Nats$ and all $x\in X$,
\[
\Prob_{x}[\tau_{A}>k \tau_{g}(\alpha)] \leq (0.1)^{k},
\]
where $\Prob$ is a product measure that generates $\{g(x,1,\cdot)\}_{x\in X}$.  
\end{lemma}

\begin{proof}
Pick a constant $\alpha>0$.
By definition of $\tau_{g}(\alpha)$, we have for $\ell \in \Nats$, $x \in X$ and $A \in \BorelSets A$ with $\pi(A) \geq \alpha$ that
\[
P_{x}[\tau_{A}>(\ell + 1)\tau_{g}(\alpha) | \tau_{A} > \ell \tau_{g}(\alpha)] \leq 0.1.
\]
Iterating this bound over $0 \leq \ell < k$, we have the desired result. 
\end{proof}

\begin{lemma} [Comparison of Lazy Hitting Times] \label{LemmaElCompLazyHit}
Let $\{g(x,1,\cdot)\}_{x\in X}$ be a transition kernel on a $\sigma$-compact metric state space $X$ with stationary distribution $\pi$. 
Let $\{g_{L}(x,1,\cdot)\}_{x\in X}$ be its associated lazy transition kernel as defined in \ref{deflazy}. 
For every $0<\alpha<\frac{1}{2}$, let $t_{H}(\alpha)$ be the maximum hitting time for the kernel $\{g(x,1,\cdot)\}_{x\in X}$, and let $\ell_{H}(\alpha)$ be the maximum hitting time for the kernel $g_{L}$. Then there exists a universal constant $0 < C < \infty$, not depending on $\alpha$, so that
\[
t_{H}(\alpha) \leq \ell_{H}(\alpha) \leq Ct_{H}(\alpha).
\]
\end{lemma}

\begin{proof}
Pick a constant $\alpha>0$. 
Fix $x \in X$ and $A \in \BorelSets X$ with $\pi(A) \geq \alpha$. Let $\{X_{t}\}_{t \in \Nats} \sim g$ with starting point $X_{0} = x$. Let $\{\zeta_{i} \}_{i \in \Nats}$ be a sequence of i.i.d. geometric random variables with mean 2, and define 
\[
L(t) = \max \{i \, : \, \sum_{j=1}^{i} \zeta_{i} \leq t\}.
\]
Then the chain $\{Y_{t}\}_{t \in \Nats}$ given by the formula $Y_{t} = X_{L(t)}$ satisfies $Y_{0} = x$ and $\{Y_{t}\}_{t \in \Nats} \sim g_{L}$. Let $\tau_{A}'$ be the first hitting time of $A$ for the chain $\{Y_{t}\}$.

We observe that 
\begin{align*}
\tau_{A}' &= \min \{t \, : \, Y_{t} \in A \} \\
&= \min \{t \, : \, X_{L(t)} \in A \} \\
&= \min \{t \, : \,  L(t) = \tau_{A} \}. \\
\end{align*}
Setting the notation 
\[
L_{\mathrm{inv}}(t) = \min \{s \, : \, L(s) = t \},
\]
this implies
\[
\tau_{A}' = L_{\mathrm{inv}}(\tau_{A}).
\]
This immediately implies $\tau_{A}' \geq \tau_{A}$ which further the first inequality in the statement of the lemma. 
On the other hand, by Markov's inequality, we have
\[
P[\tau_{A}' \geq 50 \tau_{A}] \leq P[\sum_{i=0}^{\tau_{A}} \eta_{i} \geq 50 \tau_{A}] \leq \frac{1}{25}.
\]
Let $\Prob$ be a product measure that generates $\{g(x,1,\cdot)\}_{x\in X}$.
By \ref{LemmaElSub}, we have 
\begin{align*}
\Prob[\tau_{A}' \geq 150 \tau_{g}(\alpha)] &\leq \Prob[\tau_{A}' \geq 50 \tau_{A}] + P[\tau_{A} \geq 3 \tau_{g}(\alpha)] \\
&\leq \frac{1}{25} + \frac{1}{1000} < 0.1.
\end{align*}
Let $\tau'_{g}(\alpha)$ denote the large hitting time for $\{g_{L}(x,1,\cdot)\}_{x\in X}$. 
Then we have $\tau'_{g}(\alpha)\leq 150\tau_{g}(\alpha)$. 
By \ref{maxlarge}, there exists a universal constant $C$, not depending on $\alpha$, such that $\ell_{H}(\alpha)\leq Ct_{H}(\alpha)$.
\end{proof}

We quote the following lemma from \citep{anderson2019mixavg}.
\begin{lemma}[{\citep[][Lemma.~5.1]{anderson2019mixavg}}]\label{emulti}
For every $0<\epsilon_{1}<\epsilon_{2}<\frac{1}{2}$, there exists a positive universal constant $c_{\epsilon_{1}\epsilon_{2}}$ such that
\[
t_{m}(\epsilon_{2})\leq t_{m}(\epsilon_{1})\leq c_{\epsilon_{1}\epsilon_{2}}t_{m}(\epsilon_{2})
\]
for every Markov process with unique stationary distribution. 
\end{lemma}

\begin{lemma} [Comparison of Lazy Mixing Times] \label{LemmaElCompLazyMix}
Let $\{g(x,1,\cdot)\}_{x\in X}$ be a transition kernel on a $\sigma$-compact metric state space $X$ with stationary distribution $\pi$. 
Let $\{g_{L}(x,1,\cdot)\}_{x\in X}$ be its associated lazy transition kernel as defined in \ref{deflazy}. 
For every $0<\epsilon<\frac{1}{2}$, let $t_{m}(\epsilon)$ be the mixing time for the kernel $\{g(x,1,\cdot)\}_{x\in X}$, and let $t_{L}(\epsilon)$ be the mixing time for the lazy kernel $g_{L}$. Then there exists a constant $0 < C_{\epsilon} < \infty$,  depending only on $\epsilon$, so that
\[
t_{L}(\epsilon) \leq C_{\epsilon}\, t_{m}(\epsilon).
\]
\end{lemma}

\begin{rem}
In contrast to \ref{LemmaElCompLazyHit}, the reverse inequality is not true; $t_{L}$ may be \textit{much} smaller than $t_{m}$.
\end{rem}

\begin{proof}
Let $t_{0} = \max(2t_{m}(\frac{1}{2} \epsilon), \lceil 10 \epsilon^{-1} \rceil)$. Then for any $x \in X$,
\begin{align*}
\| g_{L}(x,t_{0},\cdot) - \pi(\cdot) \| &= \| \sum_{s=0}^{t_{0}} 2^{-t_{0}} {t_{0} \choose s} g(x,s,\cdot) - \pi(\cdot) \| \\
&\leq \sum_{s=0}^{t_{0}} 2^{-t_{0}} {t_{0} \choose s} \| g(x,s,\cdot) - \pi(\cdot) \| \\
&\leq \sum_{s=0}^{\frac{1}{2} t_{0}} 2^{-t_{0}} {t_{0} \choose s} + \sum_{s=\frac{1}{2} t_{0}}^{ t_{0}} 2^{-t_{0}} {t_{0} \choose s} \, \| g(x,s,\cdot) - \pi(\cdot) \| \\
&\leq \frac{\epsilon}{2} + \frac{\epsilon}{2} \leq \epsilon.
\end{align*}
Thus, $t_{L}(\epsilon) \leq max(2t_{m}(\frac{1}{2} \epsilon), 10 \epsilon^{-1})$. Applying \ref{emulti} completes the proof.
\end{proof}

\begin{lemma} [Perturbation Bound on Mixing Times] \label{LemmaElPertMixing}
Let $\{g(x,1,\cdot)\}_{x\in X}, \{g'(x,1,\cdot)\}_{x\in X}$ be two transition kernels on a $\sigma$-compact metric state space $X$ with stationary distributions $\pi$ and $\pi'$, respectively. 
Let $t_{m}$, $t_{m}'$ be the mixing time of $g,g'$ respectively. If
\[ \label{IneqAsPert1}
\sup_{x} \| g(x,1,\cdot) - g'(x,1,\cdot) \| \leq \frac{1}{256 t_{m}},
\]
then\footnote{Note that the assumption treats $g,g'$ asymmetrically, but the conclusion is symmetric. This is not an accident.}
\[
\frac{1}{4} t_{m}' \leq t_{m} \leq 4 t_{m}'.
\]

\end{lemma}

\begin{proof}

We will show that $t_{m}' \leq 4 t_{m}$ under the assumption that 
\[ \label{IneqAsPert2}
\sup_{x} \| g(x,1,\cdot) - g'(x,1,\cdot) \| \leq \frac{1}{64 t_{m}},
\]
a slight weakening of \eqref{IneqAsPert1}. We first show that $\| \pi - \pi' \|$ is small. To do this, we make the following small extension of our notation: for any probability measure $\mu$, time $t \in \Nats$ and set $A \in \BorelSets X$, we define 
\[
g(\mu,t,A) = \int g(x,t,A) \mu(dx),
\]
and similarly for $g'$. Next, set $t_{0} = 4 t_{m}$. We have 
\begin{align*}
\| \pi - \pi' \| &= \| \pi(\cdot) - g'(\pi', t_{0}, \cdot) \| \\
&\leq \| \pi(\cdot) - g(\pi', t_{0}, \cdot) \| + \| g(\pi',t_{0},\cdot) - g'(\pi',t_{0},\cdot) \| \\
&\leq 2^{-4} + t_{0} \sup_{x \in X} \| g(x,1,\cdot) - g'(x,1,\cdot) \| \\
&\leq 2^{-4} + \frac{4}{64}.
\end{align*}

Thus,
\[\label{IneqStatPertBd}
\| \pi - \pi' \|  \leq \frac{1}{8}.
\]

We now prove our main bound. For all $x \in X$, we have 
\begin{align*}
\| g'(x,t_0,\cdot) - \pi'(\cdot) \| &= \| (g'(x,t_0,\cdot)- g(x,t_0,\cdot)) + (g(x,t_0,\cdot) - \pi(\cdot)) + (\pi(\cdot) - \pi'(\cdot)) \| \\
&\leq \| g'(x,t_0,\cdot)- g(x,t_0,\cdot) \| + \| g(x,t_0,\cdot) - \pi(\cdot) \| + \| \pi - \pi' \| \\ 
&\leq t_0 \, \sup_{y \in X} \| g'(y,1,\cdot) - g(y,1,\cdot) \| +  \| g(x,t_0,\cdot) - \pi(\cdot) \| + \frac{1}{8} \\
&\leq \frac{4}{64} + 2^{-4} + \frac{1}{8} = \frac{1}{4},
\end{align*}
where the inequality in the second-last line is Inequality \eqref{IneqStatPertBd}, the first inequality in the final line comes from assumption \eqref{IneqAsPert2}, and the second inequality in that line follows from \ref{mixequal} and \ref{submulti}. This shows that $t'_{m} \leq 4 t_{m}$, proving the first half of the inequality in the statement of the lemma.

However, we note that the inequalities $t'_{m} \leq 4 t_{m}$, combined with \eqref{IneqAsPert1}, show that 
\[ 
\sup_{x} \| g(x,1,\cdot) - g'(x,1,\cdot) \| \leq \frac{1}{64 t_{m}'}.
\]
This is exactly the weakened assumption \eqref{IneqAsPert2} with the roles of $g$ and $g'$ swapped, and so we conclude that $t_{m}\leq 4 t'_{m}$ as well.
\end{proof}

\begin{lemma} [Perturbation Bound on Hitting Times] \label{LemmaElPertHitting}
Let $\{g(x,1,\cdot)\}_{x\in X}, \{g'(x,1,\cdot)\}_{x\in X}$ be two transition kernels on a $\sigma$-compact metric state space $X$ with stationary distributions. 
Fix some constant $0 < \alpha < 0.5$. 
Let $t_{H}(\alpha)$, $t_{H}'(\alpha)$ be the maximum hitting times of $\{g(x,1,\cdot)\}_{x\in X},\{g'(x,1,\cdot)\}_{x\in X}$ respectively.
If 
\[ \label{IneqAsPert21}
\sup_{x} \| g(x,1,\cdot) - g'(x,1,\cdot) \| \leq \frac{1}{30 t_{H}(\alpha)},
\]
then
\[
t_{H}'(\alpha) \leq 60t_{H}(\alpha)
\]

\end{lemma}

\begin{proof}

For probability measures $\mu, \nu$ on measure space $(\Omega,\mathcal{F})$, we denote by $\mu \otimes_{\max} \nu$ the usual ``maximal" coupling (see Theorem 2.12 of \cite{den2012probability} for a precise construction). 
This coupling has the property that if $(X,Y) \sim \mu \otimes_{\max} \nu$, then $X \sim \mu$, $Y \sim \nu$, and furthermore
\[
\mu\otimes_{\max}\nu[X \neq Y]=\| \mu - \nu \|.
\]
By a slight abuse of notation, we use this to define the kernel $g \otimes_{\max} g'$ by the formula 
\[
(g \otimes_{\max} g')((x,y),1,\cdot)= g(x,1,\cdot) \otimes_{\max} g'(y,1,\cdot).
\]
Let $\Prob$ be a product measure on $(X\times X)^{\Nats}$ that generates $g\otimes_{\max} g'$.

We now fix $x \in X$ and sample $\{(X_{t}, X_{t}')\}_{ t \in \Nats} \sim (g \otimes_{\max} g')$ with $(X_{0},X_{0}') = (x,x)$. Let $\tau = \min \{t \, : \, X_{t} \neq X_{t}'\}$. Fix a set $A \in \BorelSets X$ and let $\tau_{A}, \tau_{A}'$ be the hitting times of $A$ for $\{X_{t}\}_{t \in \Nats}$ and $\{X_{t}'\}_{t \in \Nats}$ respectively.

Fix $0<\alpha<\frac{1}{2}$.
Let $\tau_{g}(\alpha),\tau'_{g}(\alpha)$ denote the large hitting times of $\{g(x,1,\cdot)\}_{x\in X}$ and $\{g'(x,1,\cdot)\}_{x\in X}$, respectively.
Let $P$ be a product measure that generates $\{g'(x,1,\cdot)\}_{x\in X}$. 
Fix $t \geq 3 \tau_{g}(\alpha)$.  
By \ref{maxlarge}, we have $t\geq 1.5t_{H}(\alpha)$.
We have 
\begin{align*}
P_{x}[\tau_{A}' > t] &=\Prob_{x,x}[\tau_{A}'>t]\\
\leq \Prob_{x,x}[\tau_{A} > t] + \Prob_{x,x}[\tau \leq t] \\
&\leq 0.01 + t \sup_{y \in X} \| g(y,1,\cdot) - g(y,1,\cdot) \| \\
&\leq 0.01 + 0.05 < 0.1.
\end{align*}
Since $A$, $x$ were arbitrary, we have $\tau'_{g}(\alpha)\leq 3\tau_{g}(\alpha)$. 
By \ref{maxlarge}, we have 
\[
t'_{H}(\alpha)\leq 2\tau'_{g}(\alpha)\leq 6\tau_{g}(\alpha)\leq 60t_{H}(\alpha).
\]
\end{proof}

\subsection{Properties of the Trace} \label{SubsecTraceProp}

We check that the ``trace chain of the lazy chain" and the ``lazy chain of the trace chain" are the same. To help, we introduce some notation. Define $\lak$ to be the map that takes a kernel $g$ to the lazy version $g_{L}$ defined in \ref{deflazy}. For a set $S$, define $\trk$ to be the map that takes a kernel $g$ to the trace $g^{(S)}$ of $g$ on $S$ defined in \ref{DefTraceChain}. We then have:

\begin{lemma} \label{LemmaTraceLazyCommute}
Let $\{g(x,1,\cdot)\}_{x\in X}$ be an ergodic transition kernel on a $\sigma$-compact metric state space $X$ endowed with Borel $\sigma$-algebra $\BorelSets X$.
Let $\pi$ be its stationary distribution and let $S \in \BorelSets X$ have $\pi(S) > 0$. Then 
\[
\lak(\trk(g)) = \trk(\lak(g)).
\]

\end{lemma}

\begin{rem}
This lemma justifies the choice of notation $g_{L}^{(S)}$ for the common transition kernel.
\end{rem}

\begin{proof}

We will actually prove a slightly stronger statement by constructing Markov chains from the relevant kernels on the same probability space.
Fix $x \in X$ and let $\{X_{t}\}_{t \in \Nats} \sim g$ with starting point $X_{0} = x$. We recall that the ``trace" process is defined in \ref{DefTraceChain} by transforming the time-coordinate of $\{X_{t}\}_{t \in \Nats}$ according to the sequence of ``entrance times" $\{\eta_{i}\}_{i \in \Nats}$. In Remark \ref{RemLazyAlt}, we pointed out that the ``lazy" kernel from \ref{deflazy} can be defined in terms of a similar transformation of the time-coordinate: if $\{\zeta_{i} \}_{i \in \Nats}$ is a sequence of i.i.d. geometric random variables with mean 2, and
\[\label{ezlform}
L(t) = \max \{i \, : \, \sum_{j=1}^{i} \zeta_{j} \leq t\},
\]
then the chain $\{X_{t}^{(L)}\}_{t \in \Nats}$ given by the formula 
\[\label{EqLazyRep}
X_{t}^{(L)} = X_{L(t)}
\] 
satisfies $X_{0}^{(L)} = x$ and $\{X_{t}^{(L)}\}_{t \in \Nats} \sim g_{L}$. Thus, the sequence $\{\zeta_{i}\}_{i \in \Nats}$ can be used to transform a sample from $g$ into a sample from $g_{L} = \lak{g}$, in much the same way that the sequence $\{\eta_{i}\}_{i \in \Nats}$ can be used to transform a sample from $g$ into a sample from $g^{(S)} = \trk{g}$.

Next, denote by $\{\eta_{i}^{(L)}\}_{i \in \Nats}$ the sequence of entrance times  associated with $\{X_{t}^{(L)}\}_{t \in \Nats}$ by  \ref{DefTraceChain}.  Then define the chain $\{X_{t}^{(L,S)}\}_{t \in \Nats}$ by the formula
\[
X_{t}^{(L,S)} = X_{\eta_{i}^{(L)}},
\]
so that $\{X_{t}^{(L,S)} \}_{t \in \Nats} \sim \trk(\lak(g))$. 

On the other hand,  we observe that
\[
\eta_{i}^{(L)} = L(\eta_{i}),
\]

so that we also have 
\[
X_{t}^{(L,S)}=X_{\eta_{i}^{(L)}} = X_{L(\eta_{i})},
\]
which means (by our new representation \eqref{EqLazyRep} of the lazy chain in terms of the sequence $\{\zeta_{i}\}_{i \in \Nats}$) that $\{X_{t}^{(L,S)}\}_{t\in \Nats} \sim \lak(\trk(g))$.

Since we have both $\{X_{t}^{(L,S)}\}_{t\in \Nats} \sim \lak(\trk(g))$ and $\{X_{t}^{(L,S)}\}_{t\in \Nats} \sim \trk(\lak(g))$, and the initial point $x \in S$ was arbitrary, this implies $\trk(\lak(g)) = \lak(\trk(g))$. 

\end{proof}

We also show that the trace inherits the strong Feller property:

\begin{proof} [Proof of Lemma \ref{tracedsf}]
Consider two starting points $x,y \in S$. We will denote by $\{X_{t}\}_{t \in \Nats}$, $\{Y_{t}\}_{ t \in \Nats}$ two Markov chains with transition kernel $g$ and starting points $X_{0} = x$, $Y_{0} = y$. It is straightforward to check that there exists a coupling of these two chains so that
\[
\mathbb{P}[X_{1} \neq Y_{1}] \leq \| g(x,1,\cdot) - g(y,1,\cdot) \|
\]
and also 
\[
\forall \, t \geq \eta, \, \, X_{t} = Y_{t},
\]
where 
\[
\eta = \min \{ t \geq 0 \, : \, X_{t} = Y_{t}\}
\]
is the first meeting time. We assume the chains are coupled in this way. Next, define the random times
\[
\eta^{x} = \min \{t > 0 \, : \, X_{t} \in S\}, \, \eta^{y} =   \min \{t > 0 \, : \, Y_{t} \in S\}.
\]

We then have

\begin{align}
\| g^{(S)}(x,1,\cdot) - g^{(S)}(y,1,\cdot) \| &\leq  \mathbb{P}[X_{\eta^{x}} \neq Y_{\eta^{y}}] \\
&\leq \mathbb{P}[\eta > \min(\eta^{x}, \eta^{y})] \\
&\leq \mathbb{P}[\eta > 1] \\
&= \| g(x,1,\cdot) - g(y,1,\cdot) \|.
\end{align}

Since $g$ satisfies Section \ref{assumptiondsf}, this immediately implies $g^{(S)}$ does as well.
\end{proof}

\subsection{Mixing and Hitting Times of Skeleton Chains} \label{SecMixHitSkel}

\begin{proof} [Proof of Lemma \ref{LemmaASFClaim1}]
Pick $\epsilon>0$ and $k\in \Nats$. 
Recall that the function 
\[
t \mapsto \sup_{x,y \in X} \| g(x,t,\cdot) - g(y,t,\cdot) \|
\]
is a non-increasing function of $t \in \mathbb{N}$. 
Pick $t_1\in \Nats$ such that $t_1\geq \frac{1}{k}\overline{t}_{m}(\epsilon)$.  Then
\[
\sup_{x,y \in X} \| g^{(k)}(x,t_1,\cdot) -  g^{(k)}(y,t_1,\cdot) \| &= \sup_{x,y \in X} \| g(x,kt_1,\cdot) -  g(y,kt_1,\cdot) \| \\
&\leq  \sup_{x,y \in X} \| g(x,\overline{t}_{m}(\epsilon),\cdot) -  g(y,\overline{t}_{m}(\epsilon),\cdot) \| \leq \epsilon,
\]
so that $t_{1} \geq \overline{t}_{m}^{(k)}(\epsilon)$ as well. Thus, we have $\overline{t}_{m}^{(k)}(\epsilon)\leq \lceil \frac{1}{k} \overline{t}_{m}(\epsilon)\rceil$. 

Conversely, pick $t_2\in \Nats$ such that $t_2\geq k\overline{t}_{m}^{(k)}(\epsilon)$. Then
\[
\sup_{x,y\in X}\|g(x,t_2,\cdot)-g(y,t_2,\cdot)\|\leq \sup_{x,y\in X}\|g(x,k\overline{t}_{m}^{(k)}(\epsilon),\cdot)-g(y,k\overline{t}_{m}^{(k)}(\epsilon),\cdot)\|\leq \epsilon.
\]
Thus, we have $\overline{t}_{m}(\epsilon)\leq k \overline{t}_{m}^{(k)}(\epsilon)$, completing the proof. 
\end{proof}

\begin{proof} [Proof of Lemma \ref{LemmaASFClaim2}]
Pick $\alpha>0$.
Fix a measurable set $A \subset X$ with $\pi(A) \geq \alpha$. Define $C = \lceil -\log_{2}(\alpha) + 1\rceil$ and $T=C\overline{t}_{m}$. By \ref{submulti}, we have
\[\label{bdtotal}
\sup_{x,y\in X}\parallel g(x,T,\cdot)-g(y,T,\cdot) \parallel \leq \frac{\alpha}{2}.
\]
By \ref{mixequal}, we have $\sup_{x\in X}\parallel g(x,T,\cdot)-\pi(\cdot)\parallel\leq \frac{\alpha}{2}$. Hence, we have $g(x,t,A)\geq \frac{\alpha}{2}$ for every $t\geq T$ and every $x\in X$. 
Fix $x_0\in X$, for all $\ell \in \Nats$, we have
\begin{align*}
\mathbb{P}_{x_0}^{(k)}[\tau_{A} > (\ell +1) \lceil \frac{T}{k} \rceil | \tau_{A} > \ell \lceil \frac{T}{k} \rceil ] &\leq \mathbb{P}_{x_0}[\tau_{A} > k(\ell +1) \lceil \frac{T}{k} \rceil  | \tau_{A} >k\ell \lceil \frac{T}{k} \rceil ]\\
&\leq (1 - \frac{\alpha}{2}).
\end{align*}
Iterating over $\ell \in \Nats$, this implies that
\[
\mathbb{P}_{x_0}^{(k)}[\tau_{A} > \ell \lceil \frac{T}{k} \rceil  ] \leq (1 - \frac{\alpha}{2})^{\ell}.
\]
Let $\ell_0=\lceil \frac{\log(10)}{\log(1 - \frac{\alpha}{2})}\rceil$, this implies
\[\label{skellmixbound}
\mathbb{P}_{x_0}^{(k)}[\tau_{A}>\ell_{0}\lceil \frac{T}{k} \rceil] \leq 0.1.
\]

By \ref{skellmixbound}, we have $\tau_{g}^{(k)}(\alpha)\leq \ell_{0}\lceil \frac{T}{k} \rceil$.
By \ref{maxlarge}, we have $t_{H}^{(k)}(\alpha)\leq 2\ell_{0}\lceil \frac{T}{k} \rceil$.
Since this does not depend on the choice of $x$ or $A$, we have the desired result. 
\end{proof}

\subsection{Well-Known Bounds}\label{wellknownbound}

\begin{lemma} \label{LemmaEasyDirection}
Let $0 < \alpha<\frac{1}{2}$.
Let $\mathcal{D}$ denote the collection of discrete time transition kernels with a stationary distribution on a $\sigma$-compact metric state space. 
Then there exists a universal constant $d'_{\alpha}$ such that, for every $\gkernel \in \mathcal{D}$, we have
\[
d'(\alpha)t_{H}(\alpha) \leq t_{L}.
\]
\end{lemma}

\begin{proof}

Let $\ell_{H}(\alpha)$ denote the maximum hitting time of the lazy chain.
By \ref{LemmaElCompLazyHit}, we have $t_{H}(\alpha)\leq \ell_{H}(\alpha)$.
Thus, it is sufficient to show that there exists a universal constant $d'_{\alpha}$ such that, for every $\gkernel \in \mathcal{D}$, we have
\[
d'(\alpha)\ell_{H}(\alpha)\leq t_{L}.
\]

Fix a starting point $x \in X$ and measurable set $A \subset X$ with $\pi(A) \geq \alpha$. Define $C = \lceil -\log_{2}(\alpha) + 1\rceil$ and $T=Ct_{L}$. By \ref{submulti}, we have
\[\label{bdtotal}
\sup_{x,y\in X}\parallel g_{L}(x,T,\cdot)-g_{L}(y,T,\cdot) \parallel \leq \frac{\alpha}{2}.
\]
By \ref{mixequal}, we have $\sup_{x\in X}\parallel g_{L}(x,T,\cdot)-\pi(\cdot)\parallel\leq \frac{\alpha}{2}$.

Let $\Prob^{(L)}$ denote a probability measure on $(X^{\Nats},{\BorelSets X}^{\Nats})$ that satisfies \ref{prodeq} for the lazy chain.
For all $k \in \mathbb{N}$, we have
\[ \label{IneqToIterLater}
\mathbb{P}^{(L)}_{x}[\tau_{A} > kT \, | \, \tau_{A} \geq (k-1)T] &\leq 1 - \inf_{y\in X}g_{L}(y,T,A).
\]
\begin{claim}
$1 - \inf_{y\in X}g_{L}(y,T,A)\leq 1-\frac{\alpha}{2}$.
\end{claim}
\begin{proof}
Pick $x_0\in X$.
We have
\[
&g_{L}(x_0,T,A)+\sup_{y\in X}\parallel g_{L}(y,T,\cdot)-\pi(\cdot) \parallel\\
&\geq g_{L}(x_0,T,A)+ |g_{L}(x_0,T,A)-\pi(A)|\geq \pi(A).
\]
As our choice of $x_0$ is arbitrary, we have
\[
\inf_{y\in X}g_{L}(y,T,A)\geq \pi(A)-\sup_{y\in X}\parallel g_{L}(y,T,\cdot)-\pi(\cdot) \parallel.
\]
Thus, by \ref{bdtotal}, we have
\begin{align*}
&1 - \inf_{y\in X}g_{L}(y,T,A)\\
&\leq 1-(\pi(A)-\sup_{y} \|g_{L}(y,T,\cdot)-\pi(\cdot)\|)\\
&\leq 1 - \frac{\alpha}{2}.
\end{align*}
\end{proof}
Iteratively applying Inequality \eqref{IneqToIterLater}, this claim implies that for all $k \in \mathbb{N}$
\[
\mathbb{P}^{(L)}_{x}[\tau_{A} > kT ] \leq (1 - \frac{\alpha}{2})^{k}.
\]

Let $k_0=\lceil \frac{\log(10)}{\log(1 - \frac{\alpha}{2})}\rceil$, this implies
\[\label{llmixbound}
\mathbb{P}^{(L)}_{x}[\tau_{A}>k_0T ] \leq 0.1.
\]

Let $\tau_{g}^{(L)}(\alpha)$ denote the large mixing time of the lazy chain.
By \ref{llmixbound}, we have $\tau_{g}^{(L)}(\alpha)\leq k_0T$.
By \ref{maxlarge}, we have $\ell_{H}(\alpha)\leq 2k_0T$.
Since this does not depend on the choice of $x$ or $A$, we have completed the proof with $d'(\alpha) = \frac{1}{2Ck_0}$.

\end{proof}
%
%
%
%
%
%
%
%
%
%
%
%


\section{Other Chains Without Strong Feller Condition} \label{AppOtherExt}

We give two additional results following Section \ref{statapp}.

\subsection{Almost Strong Feller Condition for Metropolis-Hastings Samplers}\label{SecASFMH}

The approach in \ref{SecASFGibbs} gives essentially the same results for ``typical" Metropolis-Hastings algorithms. Recall the definition of the Metropolis-Hastings chain in Definition \ref{EqDefMH}, which is used throughout this section. We have: 

\begin{lemma} \label{LemmaMHAsf}
Let $\mathcal{A}$ be the collection of transition kernels of the form given in \ref{gtranform} with finite mixing time, and for which $q_{x}(y)$ is uniformly continuous jointly in $x,y$. Then, for all $0<C<\infty$, there exists a universal constant $K_{C}$ so that all $\{g(x,1,\cdot)\}_{x\in X}\in \mathcal{A}$ are $(k,C)$-almost strong Feller for all $k \geq K_{C} \min(\gamma^{-1},t_{L}) \log(t_{L})$.
\end{lemma}

\begin{proof}
We mimic the proof of \ref{LemmaGibbsSamplersAsf}. Throughout this proof, we denote by $g$ a generic element of $\mathcal{A}$ and use notation from \ref{EqDefMH} freely. Fix $\delta > 0$ and let $x$ be a point satisfying $g(x,1,\{x\}^{c}) \leq \gamma + \delta$. Then for all $t \in \Nats$, we have
\begin{align*}
\| g(x,t,\cdot) - \pi \| &\geq g(x,t,\{x\}) \\
&\geq (1 - \gamma - \delta)^{t}.
\end{align*}
By the representation for the lazy chain in \ref{RemLazyAlt}, we also have 
\begin{align*}
\| g_{L}(x,t,\cdot) - \pi \| &\geq g_{L}(x,t,\{x\}) \\
&\geq (1 - \gamma - \delta)^{t}.
\end{align*}
Thus,
\[
t_{m}, t_{L} \geq \frac{\log(3)}{-\log(1 - \gamma - \delta)}.
\]
Since $\delta > 0$ is arbitrary, this implies
\[ \label{IneqMHAsfcKeyBd}
t_{m}, t_{L} \geq \frac{\log(3)}{-\log(1 - \gamma)}.
\]
Since we assume that $t_{m} < \infty$, this implies we must have $\gamma < 1$. Next, denote by $k \in \Nats$ a constant to be fixed later. Let $\{X_{t}\}_{t \in \Nats} \sim g$, and let $L$ be the (random) function from Equation \eqref{ezlform}, so that $\{X_{L(t)}\}_{t \in \Nats} \sim g_{L}$.

For this choice, define the event 
\[
\mathcal{E}^{c} = \{ X_{0} = X_{1} = \ldots = X_{L(k)} \},
\]
and define the kernels $G_{1}, G_{2}$ by setting
\begin{align*}
G_{1}((x,1,A) &= \mathbb{P}_{x}[X_{L(k)} \in A | \mathcal{E}] \\
G_{2}(x,1,A) &= \mathbb{P}_{x}[X_{L(k)} \in A | \mathcal{E}^{c}].
\end{align*}

In the notation of Definition \ref{DefASF}, the constant $p$ associated with this choice of $k, G_{1}, G_{2}$ is
\[
p = P[\mathcal{E}].
\]
By the definition of $\gamma$, we have
\[
p \leq (1 - \frac{\gamma}{2})^{k}.
\]

Comparing this to Inequality \eqref{IneqMHAsfcKeyBd}, we conclude (by the same argument as in the end of \ref{LemmaGibbsSamplersAsf}) that there exists some particular choice $K = K_{C} \in \Nats$ depending only on $C$ so that $p \leq \frac{1}{\asfc}$ for all $k \geq K_{C} \log(t_{L}) \min(\gamma^{-1}, t_{L})$.

Next, we observe that $G_{1}$ is reversible. From its definition, we see that $G_{2}$ is the kernel of the ``trivial" Markov chain that never moves (that is, $G_{2}(x,1,A)$ is just the indicator function of the set $\{x \in A\}$). Thus, $G_{2}$ is reversible for \textit{any} measure, and in particular it is $\pi$-reversible. Since $g_{L}^{(k)}$ is $\pi$-reversible and $G_{2}$ is $\pi$-reversible, this implies that $G_{1}$ is also $\pi$-reversible.

Finally, recall by our assumption that $q_{x}(y)$  is jointly uniformly continuous in its arguments and that $\rho$ is continuous. This implies that the collection $\{g^{*}(x,\cdot)\}_{x \in X}$ of sub-probability measures given by the formula 
\[
g^{*}(x,A) = g(x,1,A \backslash \{x\})
\]
satisfies Section \ref{assumptiondsf}, and thus that $G_{1}$ satisfies Section \ref{assumptiondsf} as well. Hence, we have the desired result. 
\end{proof}

We then have the following result:

\begin{thm} \label{MhThmMainCorr}
Let $\mathcal{A}$ be the collection of transition kernels of the form given in \ref{gtranform} with finite mixing time, and for which $q_{x}(y)$ is jointly continuous in $x,y$. Then there is a universal constant $0 < K < \infty$ and, for all $0 < \alpha < 0.5$, two additional constants $0 < c_{\alpha} < \infty$, $0<c_{\alpha}'<\infty$ so that for all $\{g(x,1,\cdot)\}_{x\in X}\in \mathcal{A}$, 
\[
c_{\alpha} T_{m} \leq  \ell_{H}^{(k)}(\alpha) \leq c_{\alpha}'(1 + \frac{t_{L}}{k})
\]
for all $k > K \, \min(\gamma^{-1}, t_{L}) \log(t_{L})$.
\end{thm}

\begin{proof}
The first inequality follows immediately from \ref{ThmAsfMainBd} and  \ref{LemmaMHAsf}. The second follows from \ref{LemmaASFClaim2} and \ref{mixequivalent}.
\end{proof}

\subsection{Perturbations of Reversible Chains} \label{SecPertBd}

We now show that our main conclusion also holds for chains that are sufficiently close to those that satisfy the assumptions of Theorem \ref{mixhit}. In the following, let $t_{m}, t_{L}, t_{H}(\alpha)$ denote the mixing times, lazy mixing times and maximum hitting times of the kernel $g$ as in the statement of Theorem \ref{mixhit}, and let $t_{m}', t_{L}', t_{H}'(\alpha)$ denote the analogous quantities for the kernel $g'$:

\begin{thm}\label{mixhitpert}
Let $\alpha<\frac{1}{2}$.
Then there exist universal constants $0<f^{(1)}_{\alpha},f^{(2)}_{\alpha}, f^{(3)}_{\alpha} < \infty$ such that, for every $\{g(x,1,\cdot)\}_{x\in X}\in \mathcal{M}$ and every ergodic transition kernel $g'$ satisfying 
\[
\sup_{x \in X} \| g(x,1,\cdot) - g'(x,1,\cdot) \| \leq f^{(1)}_{\alpha} t_{m}^{-1},
\]
we have
\[
f^{(2)}_{\alpha}t_{H}'(\alpha)\leq t_{L}' \leq f^{(3)}_{\alpha}t_{H}'(\alpha).
\]
\end{thm}

\begin{rem}
As in the proof of \ref{ThmAsfMainBd}, we use the notation ``$x \lesssim y$" as shorthand for the longer phrase ``there exists a universal constant $D$  such that $x \leq D y$," and we use $x \sim y$ for ``$x \lesssim y$ and $y \lesssim x$." Also, denote by $\pi'$ the stationary measure of the ergodic kernel $\{g'(x,1,\cdot)\}_{x\in X}$.
\end{rem}

\begin{proof}
Note that every ergodic kernel has a stationary distribution.
By \ref{maxhitless}, $t_{H}(\alpha) \lesssim t_{m}$. Thus, there exists a universal constant $f^{(1)}_{\alpha}$ so that the conditions of \ref{LemmaElPertMixing} and \ref{LemmaElPertHitting} are satisfied for all kernels $g'$ satisfying
\[
\sup_{x \in X} \| g(x,1,\cdot) - g'(x,1,\cdot) \| \leq f^{(1)}_{\alpha} t_{m}^{-1}.
\]
As $\sup_{x\in X}\|g_{L}(x,1,\cdot)-g'_{L}(x,1,\cdot)\|\leq \sup_{x\in X}\| g(x,1,\cdot) - g'(x,1,\cdot)\|$, by \ref{LemmaElCompLazyMix}, we can assume that
\[
\sup_{x\in X}\|g_{L}(x,1,\cdot)-g'_{L}(x,1,\cdot)\|\leq \frac{1}{256t_{L}}
\]
for all kernels $g'$ satisfying 
\[
\sup_{x \in X} \| g(x,1,\cdot) - g'(x,1,\cdot) \| \leq f^{(1)}_{\alpha} t_{m}^{-1}.
\]

Fix such a kernel $g'$, by \ref{LemmaElPertMixing} and \ref{LemmaElPertHitting}, we can conclude that
\[
t_{L}' \lesssim t_{L}, \quad t_{H}(\alpha) \lesssim t_{H}'(\alpha).
\]
Applying Theorem \ref{mixhit}, this implies 
\[
t_{L}' \lesssim t_{H}'(\alpha). 
\]

Applying \ref{maxhitless} to get the reverse inequality 
\[
t_{H}'(\alpha) \lesssim t_{L}'
\]
completes the proof. 
\end{proof}


\end{document}